\documentclass[11pt]{amsart}
\usepackage[text={155mm,235mm},centering]{geometry}
\usepackage[active]{srcltx} 

\theoremstyle{plain}
\newtheorem{thm}{Theorem}[section]
\newtheorem{prop}[thm]{Proposition}
\newtheorem{lem}[thm]{Lemma}
\newtheorem{cor}[thm]{Corollary}

\theoremstyle{definition}
\newtheorem{defn/}[thm]{Definition}
\newenvironment{defn}{ \pushQED{\qed}\begin{defn/}} {\popQED\end{defn/}}

\theoremstyle{remark}
\newtheorem{rem/}[thm]{Remark}
\newenvironment{rem}{ \pushQED{\qed}\begin{rem/}} {\popQED\end{rem/}}
\newtheorem{ex/}[thm]{Example}
\newenvironment{ex}{ \pushQED{\qed}\begin{ex/}} {\popQED\end{ex/}}

\numberwithin{equation}{section}

\usepackage[mathscr]{eucal}
\usepackage[all]{xy}
\usepackage{bm}
\usepackage{amssymb,mathrsfs,mathtools}
\usepackage{enumitem}
\usepackage{latexsym}
\usepackage{tabularx}
\usepackage{graphicx}
\usepackage{pict2e}
\usepackage[pagebackref]{hyperref}
\usepackage{tensor}
\usepackage{tikz}
\usepackage{tikz-cd}
\usepackage{textcomp}
\usepackage[scr=rsfs]{mathalfa}

\definecolor{ceruleanblue}{rgb}{0.16, 0.32, 0.75}
\hypersetup{colorlinks=true,allcolors=ceruleanblue}
\allowdisplaybreaks

\newenvironment{enumeratea}{
  \begin{enumerate}[label = (\alph*)]}{
  \end{enumerate}}

\newcommand{\R}{\mathbb{R}}

\newcommand{\Z}{\mathbb{Z}}

\DeclareMathOperator{\Aut}{Aut}

\DeclareMathOperator{\dom}{dom}
\DeclareMathOperator{\Diff}{Diff}
\DeclareMathOperator{\Hol}{Hol}
\DeclareMathOperator{\hol}{hol}

\DeclareMathOperator{\germ}{germ}

\DeclareMathOperator{\id}{id}
\DeclareMathOperator{\Iso}{Iso}

\DeclareMathOperator{\Lie}{Lie}

\DeclareMathOperator{\pr}{pr}

\DeclareMathOperator{\stab}{stab}

\newcommand{\ol}[1]{\overline{#1}}
\newcommand{\ti}[1]{\tilde{#1}}

\newcommand{\cl}[1]{\mathcal{#1}}

\newcommand{\fk}[1]{\mathfrak{#1}}
\newcommand{\mbf}[1]{\mathbf{#1}}

\newcommand{\rra}{\rightrightarrows}
\newcommand{\xar}[1]{\xrightarrow[]{#1}}

\newcommand{\fiber}[2]{\tensor[_{#1}]{\times}{_{#2}}}

\newcommand{\Bot}{\text{bot}}
\newcommand{\Top}{\text{top}}
\newcommand{\Mid}{\text{mid}}

\newcommand{\define}[1]{\textbf{#1}}

\begin{document}

\title{Riemannian foliations and quasifolds}

\author{Yi Lin}

\address{Y. Lin \\
  Department of Mathematical Sciences \\
  Georgia Southern University \\
  Statesboro, GA, 30460 USA}
\email{yilin@georgiasouthern.edu}

\author{David Miyamoto}

\address{D. Miyamoto \\
  Max Planck Institute for Mathematics \\
  Vivatsgasse 7, 53111 Bonn, Germany}
\email{miyamoto@mpim-bonn.mpg.de}



\subjclass[2020]{57S25; 57R91}

\keywords{Riemannian foliation, quasifolds, Molino theory, diffeology}
\date{\today}



\begin{abstract}
It is well known that, by the Reeb stability theorem, the leaf space of a Riemannian foliation with compact leaves is an orbifold. We prove that, under mild completeness conditions, the leaf space of a Killing Riemannian foliation is a diffeological quasifold: as a diffeological space, it is locally modelled by quotients of Cartesian space by countable groups acting affinely. Furthermore, we prove that the holonomy groupoid of the foliation is, locally, Morita equivalent to the action groupoid of a countable group acting affinely on Cartesian space.
\end{abstract}

\maketitle
\tableofcontents
\section{Introduction}

A regular foliation of a manifold $M$ is a partition $\cl{F}$ of $M$ into submanifolds of a fixed dimension, called leaves, which fit together smoothly. To study the differential geometry of the leaf space $M/\cl{F}$, an immediate obstacle is whereas $M$ and $\cl{F}$ are smooth, the leaf space $M/\cl{F}$ is rarely a manifold. For instance, consider the foliation $(M,\cl{F})$ of an open M\"{o}bius strip, where $M$ is formed by gluing the vertical edges of the square $[-1, 1] \times (-1, 1)$ with a half twist, and where the leaves are the projections of horizontal lines. Then the leaf space is homeomorphic to $(-1, 1)/\Z_2$, which has a singularity at the origin.  Nevertheless, the leaf space of a foliation always admits the structure of a diffeology, a generalized smooth structure on a set introduced by Souriau \cite{S79}. Indeed, a diffeology propagates nicely to quotients of manifolds, and allows a clean and precise notion of “local model” for singular spaces. For example, this is the case for orbifolds, which are spaces that are locally modelled by quotients of Cartesian space $\R^n$ by smooth actions of finite groups. It is not difficult to show the leaf space of our foliation of the M\"{o}bius strip is diffeologically diffeomorphic  to the orbifold $(-1,1)/\Z_2$. More generally, the Reeb stability theorem implies:

\begin{thm}
  \label{thm:leaf-space-orbifold}
  If $(M, \cl{F})$ is a foliation, and every leaf is compact with finite holonomy, then $M/\cl{F}$ is an orbifold.
\end{thm}
The requirement of finite holonomy is subtle, but automatic if $\cl{F}$ has compact leaves and is Riemannian, meaning $M$ admits a Riemannian metric compatible with $\cl{F}$. In this case:
\begin{cor}
  The leaf space of a Riemannian foliation with compact leaves is an orbifold.
\end{cor}

The compactness requirement is necessary. As a counterexample, consider a Kronecker foliation induced by irrational flows on a $2$-torus. This foliation is Riemannian, but its leaf space is non-Hausdorff and is not an orbifold. Nevertheless, the leaf space of a Kronecker foliation carries the structure of a quasifold, spaces introduced by Prato in \cite{Pr99} and \cite{Pr01} in order to generalize the Delzant construction in toric geometry to simple non-rational polytopes. Quasifolds are locally modelled by quotients of Cartesian spaces $\R^n$ by smooth affine actions of countable groups, and fit naturally in the category of diffeological spaces, cf. \cite{IZP} and \cite{KM22}.

In this article, we show that the leaf spaces of Killing Riemannian foliations are diffeological quasifolds. The definition of a Killing Riemannian foliation is rather complicated (see Section \ref{Molino-theory}), so here we give three examples: Riemannian foliations on simply-connected manifolds are Killing; foliations of compact Riemannian manifolds whose leaves are orbits of connected Lie subgroups of isometries are Killing; and finally, Battaglia and Zaffran \cite{BZ17} show that every toric symplectic quasifold is equivariantly symplectomorphic to the leaf space of a Killing foliation.\footnote{More precisely, they show the relevant foliation is Riemannian, but mention that a careful reading of their proof also shows it is Killing.} We prove:

\begin{thm}
  \label{thm:leaf-space-quasifold}
  If $(M, \cl{F})$ is a complete Killing foliation of a connected manifold $M$, with complete transverse action of its structure algebra, then $M/\cl{F}$ is a diffeological quasifold.
\end{thm}
Lie groupoids give an alternative approach to transverse geometry. Specifically, we have the holonomy groupoid $\Hol(\cl{F})$ of a foliation $(M, \cl{F})$. For Killing foliations, we have:
\begin{cor}
In the setting of Theorem \ref{thm:leaf-space-quasifold}, $\Hol(\cl{F})$ is, locally, Morita equivalent to the action groupoids of countable groups acting affinely on a Cartesian space.
\end{cor}

These are Theorem \ref{thm:1} and Corollary \ref{cor:morita-equivalence} in this article, respectively. To prove these results, we show that any complete Killing foliation is locally developable, and then apply a description of leaf spaces of developable foliations worked out in Propositions \ref{prop:develop-leaf-space} and \ref{prop:morita-to-pseudogroup}. To show local develop-ability, we use Molino's structure theory for Riemannian foliations (see \cite{Mo88} and Section \ref{sec:review-riemannian-fol}), and specifically Fedida's Theorem \ref{leaf-space} for complete Lie-$\fk{g}$ foliations. We emphasize that our local models are quotients of Cartesian spaces by countable groups acting affinely, and thus we consider our result a form of “affinization”. It is distinct from the usual linearization of proper Lie group actions (or proper Lie groupoids), which are known to be locally isomorphic to the action of the isotropy groups on the normal space to an orbit. For example, in Example \ref{ex:dense-leaf}, we apply Theorem \ref{thm:1} to conclude that the leaf space of a Kronecker foliation with angle $\lambda$ is the quotient of $\R$ by the subgroup $\Z + \lambda\Z$, whereas if we attempt to linearize its holonomy groupoid like a proper groupoid, we get a trivial groupoid.

Haefliger \cite{H88} gave a model for complete pseudogroups of isometries which, when applied to the holonomy pseudogroup of a complete Killing Riemannian foliation, achieves a result of a similar spirit to our own. However, the actions in his local model are not necessarily affine, and their affinization - as supplied by our model - is not obvious. We carefully compare our result to Haefliger's in Section \ref{sec:comp-haefl-model}.

We also consider a partial converse to Theorem \ref{thm:1}. The case of orbifolds is well known. Given an (oriented) orbifold, the total space of its orthonormal frame bundle is a smooth manifold, on which the structure group acts locally freely. Furthermore, we can identify the orbifold with the orbit space of the locally free structure group action. As a result, every orbifold can be realized as the leaf space of a Riemannian foliation. We address the case of quasifolds in Section \ref{sec:diff-quas}. We show that every quasifold is a leaf space of a foliation (Corollary \ref{cor:quas-is-leaf-space}), and give a condition for that foliation to be Riemannian (Corollary \ref{cor:riemannian-quas-is-leaf-space}). We use the structural pseudogroup associated to a diffeological quasifold, and our methods differ from Battaglia and Zaffran's in \cite{BZ17}. We do not currently know the relation between our foliation and Battaglia and Zaffran's, in the case of toric symplectic quasifolds.

Our paper is structured as follows. In Section \ref{sec:quas-as-diff}, we review diffeology for unfamiliar readers, and also introduce diffeological quasifolds. In Section \ref{sec:trans-geom-structures}, we review the transverse geometry of foliations. In particular, we prove Lemma \ref{quotient-bundle}, which gives conditions for a foliated bundle's leaf space to be a manifold. In Section \ref{sec:review-riemannian-fol}, we review Molino's theory of Riemannian foliations, and we define Killing foliations. In Section \ref{sec:leaf-space-devel}, we describe the global transverse geometry of developable foliations from diffeological and Lie groupoid perspectives. In Section \ref{sec:leaf-space-killing} we prove our main result, Theorem \ref{thm:1}, and give some examples. Finally in Section \ref{sec:diff-quas}, we approach the converse direction of our main result.

\section{Quasifolds as diffeological spaces}
\label{sec:quas-as-diff}

Manifolds are spaces that are locally modelled by Cartesian space $\R^n$. Orbifolds, introduced by Satake in \cite{S56} and \cite{S57} by the name V-manifolds, are spaces locally modelled by quotients of $\R^n$ by actions of finite groups. Quasifolds, introduced by Prato in \cite{Pr99} and \cite{Pr01}, are spaces locally modelled by quotients of $\R^n$ by actions of countable groups. However, whereas the meaning of ``locally modelled'' is well established in the case of manifolds, its meaning for orbifolds and quasifolds is more subtle. Satake's formulations in \cite{S56} and \cite{S57} are different and their equivalence is non-trivial, and Prato's in \cite{Pr99} gives yet another approach. Thankfully, diffeology, first defined by Souriau \cite{S79} (though a similar notion was introduced by Chen \cite{C77}), provides a category that supports straightforward definitions of manifolds, orbifolds, and quasifolds. In this section, we introduce the necessary basics of diffeology, and diffeological orbifolds and quasifolds. Our main source for diffeology is Iglesias-Zemmour's textbook \cite{IZ13}. We refer the reader there for more detail.

A \define{diffeology} on a set $X$ is a collection of maps $\mathscr{D}$ from open subsets of Cartesian spaces into $X$, called \define{plots}, satisfying:
\begin{itemize}
\item all locally constant maps $U \to X$ are plots (for all $n \geq 0$);
\item if $p:U \to X$ is a plot, and $F:V \to U$ is smooth, then $p^*F:V \to X$ is a plot;
\item if $p:U \to X$ is a function, and is locally a plot, then $p$ is a plot.
\end{itemize}
A set equipped with a diffeology is a \define{diffeological space}.  A function $f:X \to X'$ between diffeological spaces is \define{smooth} if $p^*f$ is a plot of $X'$ whenever $p$ is a plot of $X$. A smooth function with smooth inverse is a \define{diffeomorphism}. Diffeological spaces, together with smooth maps, form a category.
\begin{ex}
  For a manifold $M$,\footnote{A Hausdorff, second-countable space equipped with a maximal smooth atlas} the collection of all smooth functions from open subsets of Cartesian spaces into $M$ is a diffeology $\mathscr{D}_M$. The assignment
  \begin{equation*}
    M \mapsto (M, \mathscr{D}_M), \quad (M \xrightarrow{f} N) \mapsto (M \xrightarrow{f} N)
  \end{equation*}
  is a fully faithful functor from the category of manifolds to that of diffeological spaces.
\end{ex}
The \define{D-topology} of a diffeological space is the finest topology for which all the plots are continuous. Smooth maps are continuous in the D-topology, so we have a faithful functor from the category of diffeological spaces to that of topological spaces.

Diffeologies pass to subsets and quotients. The \define{subset diffeology} on $A \subseteq X$ consists of the plots of $X$ with image in $A$. For an equivalence relation ${\sim}$ on $X$, the \define{quotient diffeology} on $X/{\sim}$ consists of the plots $p:U \to X/{\sim}$ that locally lift along the quotient map $\pi$. Locally lift means: for each $r \in U$, there is a neighbourhood $U'$ or $r$ and a plot $q:U' \to X$ such that $\pi \circ q = p|_{U'}$. The D-topology of the quotient diffeology is the usual quotient topology, but the D-topology of the subset diffeology is not necessarily the subspace topology.

\begin{ex}
  Suppose $(M, \cl{F})$ is a foliated manifold, and $\pi:M \to M/\cl{F}$ is the quotient map to its leaf space. We equip the leaf space with the quotient diffeology, and speak of the ``diffeological leaf space.'' If $U$ is open in $M$, then $\pi(U)$ is open in $M/\cl{F}$.
\end{ex}

We now define diffeological quasifolds. This is the definition from \cite{IZP} and \cite{KM22}.
\begin{defn}\label{def:quasifold}
  A \define{diffeological quasifold} (of dimension $n$) is a second-countable diffeological space $X$ such that: for every $x \in X$, there is a D-open neighbourhood $U$ of $x$, a countable subgroup $\Gamma$ of $\operatorname{Aff}(\R^n)$, a $\Gamma$-invariant open subset $V$ of $\R^n$, and a diffeological diffeomorphism $F:U \to V/\Gamma$. We call $F$ a \define{quasifold chart}. A collection of quasifold charts that cover $X$ is an \define{atlas} of $X$.
\end{defn}
If every group $\Gamma$ in Definition \ref{def:quasifold} is finite, we obtain a \define{diffeological orbifold}, in the sense of \cite{IZKZ10}. If every group $\Gamma$ is trivial, we obtain a not-necessarily-Hausdorff manifold. 
\begin{rem}
  The fact diffeological orbifolds are not necessarily Hausdorff distinguishes them from other orbifolds in the literature, including Satake's original V-manifolds. See \cite{IZKZ10} for a thorough comparison of existing definitions. Our models for diffeological quasifolds also differ from those used by Prato in \cite{Pr99}. However, Prato introduced quasifolds in order to extend Delzant's theorem from symplectic geometry to the case of non-rational simple polytopes. The quasifolds that appear in her extension all satisfy Definition \ref{def:quasifold}.  
\end{rem}

\begin{ex}\
  \begin{itemize}
  \item Suppose $(M, \cl{F})$ is a foliated manifold and every leaf of $\cl{F}$ is compact with finite holonomy group (see Section \ref{sec:struct-pseud-quas}). Then the diffeological leaf space $M/\cl{F}$ is a diffeological orbifold. This is a consequence of the Reeb stability theorem for foliations, see \cite[Theorem 2.15]{MM03}.

  \item Irrational tori are important examples of diffeological quasifolds. Take $\lambda$ irrational, and consider the action of the group $\Z + \lambda \Z$ on $\R$ by addition. The quotient $\R/(\Z + \lambda \Z)$ is a diffeological quasifold, which we call an irrational torus. This quasifold also arises the leaf space of the foliation on the torus $T^2 = \R^2/\Z^2$ induced by the foliation on $\R^2$ by lines of slope $\lambda$. For a proof, see \cite[Exercise 31]{IZ13}, but we also derive this fact in Example \ref{ex:dense-leaf}, as a consequence of our Theorem \ref{thm:1}.
  \end{itemize}
\end{ex}

Associated to every quasifold is a countably-generated pseudogroup encoding its structure. The language of pseudogroups will be useful throughout this article, so we review it here. A \define{transition} between manifolds (or diffeological spaces) $M$ and $M'$ is a diffeomorphism from an open subset of $M$ onto an open subset of $M'$. We may denote a transition by $\psi:M \dashrightarrow M'$. A \define{pseudogroup} on $M$ is a collection $\Psi$ of transitions of $M$ such that:
\begin{itemize}
\item it contains the identity;
\item it is closed under composition, inversion, and restriction to open subsets;
\item if $f$ is a transition on $M$, and is locally in $\Psi$, then $f$ is in $\Psi$.
\end{itemize}
We say $\Psi$ is \define{countably-generated} when there is a countable collection of transitions $\{\psi_i\}_{i=1}^\infty$ such that for every $\psi \in \Psi$ and $x \in \dom \psi$, we have $\operatorname{germ}_x\psi = \operatorname{germ}_x \psi_i$ for some $i$. This is equivalent to requiring that $\Psi$ is the smallest pseudogroup containing the $\psi_i$.

\subsection{The structural pseudogroup of a quasifold}
\label{sec:struct-pseud-quas}

In this subsection, we associate a pseudogroup to a quasifold. We will use this pseudogroup in Section \ref{sec:diff-quas}, and the reader may skip it on first reading. We rely on the results in \cite{KM22}, cf.\ \cite{IZP}.
\begin{lem}
  Let $\Gamma$ be a countable subgroup of $\operatorname{Aff}(\R^n)$, and $V$ be a $\Gamma$-invariant open subset of $\R^n$. Suppose $\psi$ is a transition of $V$ preserving $\Gamma$-orbits. Then for each $x \in \dom \psi$, there is some $\gamma \in \Gamma$ such that $\operatorname{germ}_x\psi = \operatorname{germ}_x \gamma$. 
\end{lem}
This is proved in \cite[Corollary 2.15]{KM22}, using the Baire category theorem.
\begin{cor}\label{cor:lifting}
  Take $\Gamma$ and $V$ as in the previous lemma, and similarly take $\Gamma'$ and $V'$. Fix a transition $F:V/\Gamma \dashrightarrow V'/\Gamma'$. For every $x \in V$ with $\pi(x) \in \dom F$, there is some transition $\psi: V \dashrightarrow V'$ defined near $x$ lifting $F$. Conversely, every local lift of $F$ is a local diffeomorphism. If $\psi$ and $\psi'$ are two transitions lifting $F$, and $x \in \dom \psi \cap \dom \psi'$, there is some $\gamma' \in \Gamma'$ such that $\operatorname{germ}_x\psi = \operatorname{germ}_x \gamma' \circ \psi'$.
\end{cor}
\begin{proof}
  The first claim is \cite[Lemma 2.16]{KM22}. The second claim is implicitly proved in that same lemma. To prove the last claim, observe that $\psi \circ (\psi')^{-1}$ is a transition of $V'$ preserving $\Gamma'$-orbits, so $\operatorname{germ}_{\psi'(x)} \psi \circ (\psi')^{-1} = \operatorname{germ}_{\psi'(x)} \gamma'$ for some $\gamma' \in \Gamma'$, as required.
\end{proof}
We now describe the pseudogroup associated to a quasifold, as done in \cite[Section 4]{KM22}.
\begin{prop}\label{prop:quastopseudo}
  Fix a quasifold $X$. Associated to each atlas $\mathcal{A}$ is a pseudogroup $\Psi$ of a manifold $V$ such that $V/\Psi \cong X$. If $\mathcal{A}$ is countable, then $\Psi$ is countably-generated.
\end{prop}
\begin{proof}
  Choose a countable atlas $\{F_i:U_i \to V_i/\Gamma_i\}$ of $X$. If we have an uncountable atlas, the argument proceeds identically, but we will not conclude $\Psi$ is countably-generated. Let $V := \bigsqcup V_i$, and let $\Psi$ be the pseudogroup consisting of all transitions $\psi:V_i \dashrightarrow V_j$ such that
  \begin{equation}\label{eq:definepsi}
    \begin{tikzcd}
      V_i \ar[rr, "\psi", dashed] \ar[d] & & V_j \ar[d] \\
      V_i/\Gamma_i \ar[r, "F_i"]&  X  &\ar[l, "F_j"] V_j/\Gamma_j.
    \end{tikzcd}
  \end{equation}
  In \cite[Proposition 4.5]{KM22}, we show that the quotient map $\bigsqcup V_i \to \bigsqcup V_i/\Gamma_i$ descends to a diffeomorphism $M/\Psi \to X$, giving the diagram
  \begin{equation*}
    \begin{tikzcd}
      \bigsqcup V_i \ar[d] \ar[r] & \bigsqcup V_i/\Gamma_i \ar[d] \\
      V/\Psi \ar[r, "\cong"] & X,
    \end{tikzcd}
  \end{equation*}
  where the right downward arrow is induced by the $F_i$. To show $\Psi$ is countably-generated, for each $i,j$ choose a countable collection $\{\psi_k\}_{k=1}^\infty$ so that each $\psi_k$ completes \eqref{eq:definepsi}, and the domains of the $\psi_k$ cover $\pi_i^{-1}(F_i^{-1}(U_j)) \subseteq V_i$. This is possible by Corollary \ref{cor:lifting}. Now, given some $\psi$ as in \eqref{eq:definepsi}, and $x \in \dom \psi$, take some $\psi_k$ whose domain also contains $x$. Then by Corollary \ref{cor:lifting}, there is some $\gamma \in \Gamma_j$ so that $\operatorname{germ}_x \psi = \operatorname{germ}_x (\gamma \circ \psi_k)$. It follows that that $\{\gamma \circ \psi_k \mid \gamma \in \Gamma_j\}_{k=1}^\infty$ generates all transitions of $\Psi$ from $V_i$ to $V_j$. Repeating this for all pairs $i,j$, we grow a countable generating set for $\Psi$.
\end{proof}

\section{Transverse geometric structures on foliations}
\label{sec:trans-geom-structures}
This section reviews the transverse geometry of foliations. Our main sources are \cite{MM03} and \cite{Mo88}. Let $M$ be a manifold, equipped with a regular foliation $\cl{F}$. This may be viewed as:
\begin{itemize}
\item An involutive subbundle of the tangent bundle. The maximal integral submanifolds are leaves of $\cl{F}$.
\item An atlas of charts $\phi:U \to \R^p \times \R^q$ whose change-of-charts diffeomorphisms are locally of the form $(g(x,y), h(y))$. The connected components of $\phi^{-1}(\R^p \times \{0\})$, called \define{plaques}, assemble into the leaves of $\cl{F}$.
\item A collection of submersions $s:U \to \R^q$, such that for any $s_i,s_j$ in the collection, there is a diffeomorphism (writing $U_{ij} := U_i \cap U_j)$
  \begin{equation*}
    h_{ij}:s_j(U_{ij}) \to s_i(U_{ij})
  \end{equation*}
  such that $h_{ij} \circ s_j = s_i$. The components of the fibers of the $s$ are the plaques. This is a \define{Haefliger cocycle} for $\cl{F}$. 
\end{itemize}
Throughout this note we denote by $\mathfrak{X}(\mathcal{F})\subset\mathfrak{X}(M)$ the subspace of vector fields tangent to the leaves of $\mathcal{F}$. We say a smooth function $f$ on $M$ is a \define{basic function} if
\begin{equation*}
  \mathcal{L}_Xf=0 \text{ for all } X\in \mathfrak{X}(\mathcal{F}).
\end{equation*}
Equivalently, $f$ is basic if it is constant along the leaves of $\cl{F}$. We will denote the set of basic functions on $(M, \mathcal{F})$ by $\Omega^0(M, \mathcal{F})$. We say that a vector field $X\in\mathfrak{X}(M)$ is \define{foliate} if
\begin{equation*}
  \cl{L}_X Y =[X,Y]\in \mathfrak{X}(\mathcal{F}) \text{ for all } Y\in \mathfrak{X}(\mathcal{F}). 
\end{equation*}
The flow of a foliate vector field preserves (but not necessarily fixes) the leaves of $\cl{F}$. We will denote by $\mathfrak{R}(\mathcal{F})$ the space of foliate vector fields.  Clearly $\fk{X}(\cl{F})$ forms an ideal in the Lie algebra of foliate vector fields. We define a \define{transverse vector field} to be the equivalence class of a foliate vector field in the quotient $\mathfrak{R}(\mathcal{F})/\mathfrak{X}(\mathcal{F})$, which we denote by $\mathfrak{X}(M, \mathcal{F})$. The space $\mathfrak{X}(M, \mathcal{F})$ of transverse vector fields forms a Lie algebra, with Lie bracket inherited from the natural one on $\mathfrak{R}(\mathcal{F})$.

\begin{defn}
  A \define{transverse action} of a Lie algebra $\mathfrak{g}$ on a foliated manifold $(M,\mathcal{F})$ is a Lie algebra homomorphism
  \begin{equation}\label{t-action}
    a: \mathfrak{g}\rightarrow\mathfrak{X}(M, \mathcal{F}).\
  \end{equation}
  We say the transverse Lie algebra action (\ref{t-action}) is \define{complete}, if for every $\xi\in \mathfrak{g}$, $a(\xi)$ is represented by a complete foliated vector field.
\end{defn}

\begin{ex}\label{ex:homogeneous-foliations}
  Take a compact Riemannian manifold $(M, g)$. The isometry group $\Iso(M)$ is a Lie group \cite{MS39}. Let $H$ be a connected Lie subgroup of $\Iso(M)$, whose isotropy groups are all of the same dimension. Then the orbits of the $H$ action on $M$ all have the same dimension, and thus form a foliation $\cl{F}$ of $M$. Let $K := \ol{H}$ be the closure of $H$ in $\Iso(M)$. The orbits of the $K$ action are the closures of the leaves of $\cl{F}$. The Lie group $H$ is normal in $K$, thus we have the quotient Lie algebra $\fk{g} := \Lie(K)/\Lie(H)$. The map
  \begin{equation*}
    \Lie(K) \to \fk{R}(\cl{F}), \quad \eta \mapsto \eta_M,
  \end{equation*}
  where $\eta_M$ is the fundamental vector field associated to $\eta$, descends to a transverse action of $\fk{g}$ on $(M, \cl{F})$. We will occasionally return to this example.
\end{ex}

\begin{defn} Consider a transverse Lie algebra action $a: \mathfrak{g}\rightarrow \mathfrak{X}(M, \mathcal{F})$. The \define{stabilizer} of $x\in M$ is defined to be
  \begin{equation}\label{stabilizaer} \text{stab}(\mathfrak{g}\ltimes \mathcal{F}, x) := \{\xi\in \mathfrak{g}\,\vert\, a(\xi)_x=0\}.
  \end{equation}
\end{defn}
We use the notation $\fk{g} \ltimes \cl{F}$ because the transitive Lie algebra action of $\fk{g}$ induces the \emph{transverse action Lie algebroid} over $M$, as defined in \cite{LS20}, which is denoted there by $\fk{g} \ltimes \cl{F}$.

\begin{lem}[{\cite[Lemma 2.2.7]{LS20}}] \label{leaf-wise-stabilizer}
  Suppose $x,y \in M$ are in the same leaf of $\mathcal{F}$, and we have a transverse action of $\fk{g}$ on $\cl{F}$. Then $\text{stab}(\mathfrak{g} \ltimes \mathcal{F}, x) = \text{stab}(\mathfrak{g} \ltimes \mathcal{F}, y)$.
\end{lem}

We will often restrict transverse actions to foliated submanifolds.
\begin{defn} \label{saturated}
  Let $(M, \mathcal{F})$ be a foliated manifold. We say that a subset $A\subset M$ is $\mathcal{F}$-\define{saturated} if any leaf of $M$ is either disjoint from $A$, or is contained in $A$. We call an $\cl{F}$-saturated embedded submanifold $A$ of $M$ a \define{foliated submanifold}.
\end{defn}

\begin{defn} \label{invariant}
  Suppose we have a transverse action $a: \mathfrak{g}\rightarrow \mathfrak{X}(M, \mathcal{F})$ on a foliated manifold $(M, \mathcal{F})$, and that $X$ is a foliated submanifold. We say that $X$ is \define{invariant} under the transverse action of $\mathfrak{g}$ if, for every $\xi\in \mathfrak{g}$, every foliated vector field $\xi_M$ representing $a(\xi)$ has the following property:
  \begin{center}
    If $\xi_M$ is tangent to $X$ at any point of $X$, then $\xi_M$ is tangent to $X$ at every point of $X$.
  \end{center}
\end{defn}

Let $(M, \mathcal{F})$ and $(M', \mathcal{F}')$ be two foliated manifolds, and let $f: M\rightarrow M'$ be a foliated map, meaning $f$ sends leaves of $\cl{F}$ to leaves of $\cl{F}'$. By assumption, for each $X_p \in T_p\cl{F}$, the pushforward $f_*(X_p)$ is in $T_{f(p)}\cl{F}'$. Thus $f$ naturally induces a normal derivative $(Nf)_p: N_p\mathcal{F}\rightarrow N_{f(p)}\mathcal{F}'$. Let $\zeta$ and $\zeta'$ be two transverse vector fields on $M$ and $M'$ respectively. We say that $\zeta$ and $\zeta'$ are $f$-\define{related}, if $\zeta'_{f(p)}=(Nf)_p(\zeta_p)$ for all $p\in M$. 

\begin{defn}
  Consider transverse Lie algebra actions $a: \mathfrak{g}\rightarrow \mathfrak{X}(M, \mathcal{F})$ and $a': \mathfrak{g}\rightarrow \mathfrak{X}(M', \mathcal{F}')$, of the same Lie algebra $\fk{g}$. We say a foliate map $f: M\rightarrow M'$ is $\mathfrak{g}$-\define{equivariant}, if for every $\xi \in \mathfrak{g}$, $a(\xi)$ and $a'(\xi)$ are $f$-related.
\end{defn}

A useful example to keep in mind is that of simple foliations.
\begin{defn} We say that a foliation $\mathcal{F}$ on $M$ is \define{simple} if the space of leaves $\overline{M} = M/\cl{F}$ carries a smooth manifold structure, such that the quotient map $\pi: M\rightarrow M/\mathcal{F}$ is a submersion with connected fibers.  Given an arbitrary foliation, we say an open subset $U\subset M$ is simple if the restricted foliation $\mathcal{F}\vert_U$ is simple.
\end{defn}
\begin{ex}
  A transverse Lie algebra action of $\fk{g}$ on a simple foliation $(M, \cl{F})$ is equivalent to a Lie algebra action of $\fk{g}$ on the leaf space $M/\cl{F}$. If $(M', \cl{F}')$ is another simple foliation that $\fk{g}$ acts on transversely, a foliate map $f:M \to M'$ is $\fk{g}$-equivariant if it descends to a $\fk{g}$-equivariant map $M/\cl{F} \to M'/\cl{F}'$.
\end{ex}

Next, we review the definition of a holonomy diffeomorphism, holonomy group, and holonomy pseudogroup from foliation theory, cf.\ \cite{Mo88}. Fix a foliated atlas for $(M, \cl{F})$: this is the maximal atlas of $M$ whose charts are of the form $\varphi_\alpha: U_\alpha \to \varphi_\alpha(U_\alpha) \subseteq \R^p \times \R^q$, such that the change of coordinates $\varphi_{\beta\alpha}$ take the form $\varphi_{\beta\alpha}(x,y) = (g_{\beta\alpha}(x,y), h_{\beta\alpha}(y))$, and the connected components of the $\varphi_\alpha^{-1}(\R^p \times \{y\})$ (the plaques) are the leaves of $\cl{F}|_{U_\alpha}$. Now let $L$ be a leaf of $\cl{F}$, and $\sigma: [0, 1]\rightarrow L$ be a continuous path joining $x$ and $y$. Choose $S$ and $S'$ two transversals of the foliation passing through $x$ and $y$, respectively. The \define{holonomy} of $\sigma$ is the germ of a diffeomorphism $h_{\sigma}^{S', S}$, which we may also denote $\hol^{S', S}(\sigma)$, mapping a neighbourhood of $x$ in $S$ to a neighbourhood of $y$ in $S'$, and mapping $x$ to $y$, defined as follows. Choose a finite cover $\{U_0, U_2, \ldots U_k\}$ of $\sigma([0,1])$ by foliation charts, such that $U_{i-1} \cap U_i \neq \emptyset$ for $1 \leq i \leq k$, and $x \in U_0$ and $y \in U_k$. A point $x' \in S$ near $x$ is contained in a unique plaque (leaf of $\cl{F}|_{U_0}$) of $U_0$, which we denote $P_0$. The plaque $P_0$ intersects a unique plaque $P_1$ of $U_1$; uniqueness follows from the definition of a foliated atlas. The plaque $P_1$ intersects a unique plaque $P_2$ of $U_2$, and we continue this chain to find that $x'$ determines a unique plaque $P_k$ of $U_k$. The plaque $P_k$ intersects the transversal $S'$ at a unique point $y'$, provided we chose $x'$ close enough to $x$. We define $h^{S',S}_\sigma$ to be the diffeomorphism $x' \mapsto y'$ here described.

It is explained in detail in \cite{Mo88} that the germ of $h_{\sigma}^{S', S}$ at $x$ is independent of the choice of open cover $\{U_0, \ldots, U_k\}$. In fact, it only depends on the homotopy class of $\sigma$ relative endpoints. When $x=y$ and $S=S'$, the germs of the holonomy diffeomorphisms $\{\germ_x h_{\sigma}^{S, S} \mid [\sigma]\in \pi_1(L, x)\}$ form a group, which we will denote by $\Hol_S( L, x)$. Finally, denoting by $\bar{x}$ the constant path at $x$, and noting that if $T$ is another transversal at $x$, the map
\begin{equation}\label{conjugate}
  \Hol_S(L, x)\rightarrow \Hol_T(L, x), \quad  \germ_x h_{\sigma}^{S}\mapsto  \germ_x h_{\bar{x}}^{S, T}\circ h_{\sigma}^{S} \circ (h_{\bar{x}}^{S, T})^{-1}
\end{equation}
is a group isomorphism, thus we see that up to conjugacy $\Hol_S(L, x)$ is independent of the choice of a transversal. We therefore refer to $\Hol_S(L,x)$ as the \define{holonomy group} of the leaf $L$ \define{at} $x$, and denote it simply by $\Hol(L, x)$. We say two paths, $\sigma_1,\sigma_2$ in $L$ with the same endpoints $x$ and $y$ are in the same holonomy class if the concatenation $\sigma * \sigma^{-1}$ represents the identity in $\Hol(L, x)$. To be in the same holonomy class is an equivalence relation.

If $S$ is a complete transversal to $\cl{F}$, meaning it is transverse to the leaves and meets every leaf at least once, we define the pseudogroup $\Psi_S(M, \cl{F})$ to be generated by all the holonomy diffeomorphisms $h_\sigma^{S,S}$ for leafwise paths $\sigma$ beginning and ending in $S$. If $S'$ is another complete transversal to $\cl{F}$, the pseudogroup $\Psi_{S'}(M, \cl{F})$ is equivalent to $\Psi_S(M, \cl{F})$, so up to equivalence we may speak of the \define{holonomy pseudogroup} of $\cl{F}$, denoted $\Psi(M, \cl{F})$. Because we may cover $M$ with countably many foliated charts, this pseudogroup is countably generated.

Now we discuss foliated bundles.

\begin{defn}\label{foliate-v-bundle}
  Suppose that $(M,\mathcal{F})$ is a foliated manifold, $G$ is a Lie group, and $\pi: P\rightarrow M$ is a principal $G$-bundle over $M$. We say $P$ is \define{foliated} when $P$ is equipped with a $G$-equivariant foliation $\cl{F}_P$, of the same dimension as $\mathcal{F}$, such that $T_p L_P \cap \operatorname{Ker}(d_p\pi) = \{0\}$ for all $p \in P$. We say a vector bundle $E\rightarrow M$ is foliated if its associated principal $GL(r)$-bundle is foliated.
\end{defn}
\begin{rem}
  For a foliated bundle $\pi:P \to M$, the projection $\pi$ maps the leaves of the lifted foliation $\cl{F}_P$ to leaves of $\cl{F}$. Furthermore, the restriction $\pi:L_P \to L$ is a covering map.
\end{rem}

The standard example of a foliated bundle is the normal bundle $N\cl{F}$ of a foliation, given as follows. Let $X$ be a vector field on $M$ tangent to $\cl{F}$.  Denote by $\phi_t$ the local flow on $M$ generated by $X$. For each $t$, the tangent map of $\phi_t$ induces a diffeomorphism $(\phi_t)_*: N\mathcal{F} \to N\mathcal{F}$. This gives rise to a local flow $(\phi_t)_*$ and a corresponding vector field $X_{N\cl{F}}$ on $N\cl{F}$ which projects down to $X$ under the bundle projection $\pi:N\cl{F} \to M$. In particular, the lift to $N\cl{F}$ of all vector fields tangent to $\cl{F}$ defines a foliation $\mathcal{F}_{N\mathcal{F}}$ on $N\mathcal{F}$, such that $(N\mathcal{F}, \cl{F}_{N\cl{F}})$ is a foliated vector bundle over $(M, \mathcal{F})$.

This next example is relevant in the sequel.
\begin{ex}\label{ex:lifted-foliation} \
  Let $(M,\mathcal{F})$ be a foliated manifold, and let $A$ be a foliated submanifold. Suppose that $X$ is a vector field on $M$ tangent to $\cl{F}$. Denote by $\phi_t$ the local flow on $A$ generated by $X|_A$. For each $t$, the tangent map of $\phi_t$ induces a diffeomorphism $(\phi_t)_*: NA\rightarrow NA$. This gives rise to a local flow $(\phi_t)_*$ on $NA$ and a corresponding vector field $X_{NA}$ on $NA$ which projects down to $A$ under the bundle projection $\pi:NA \to A$. If we lift all vector fields tangent to $\mathcal{F}$ in this manner, we get a foliation $\mathcal{F}_{NA}$ on $NA$, such that $(NA, \cl{F}_{NA})$ is a foliated vector bundle over $A$. In addition, if we have a transverse $\mathfrak{g}$-action on $(M, \mathcal{F})$ that leaves $A$ invariant, it lifts to a transverse $\mathfrak{g}$-action on $(NA, \cl{F}_{NA})$.
\end{ex}

Let $X$ be a $\mathfrak{g}$-invariant embedded submanifold of $M$, and let $\mathcal{F}_X := \mathcal{F}\vert_X$ be the restriction of the foliation to $X$. Then by Example \ref{ex:lifted-foliation}, the normal bundle
\begin{equation*}
  NX =TM\vert_X /TX\cong N\mathcal{F}/N\mathcal{F}_X
\end{equation*}
is a foliated vector bundle over $(X, \mathcal{F}_X )$ and is equipped with a natural transverse $\mathfrak{g}$-action with the property that the bundle projection $N X \rightarrow X$ is equivariant.

\begin{defn}
  \label{def:tubular-neighbourhood}
  A $\mathfrak{g}$-invariant \define{tubular neighbourhood} of X is a $\mathfrak{g}$-equivariant foliate embedding $\phi : NX \hookrightarrow M$ with the following properties:
  \begin{itemize}
  \item the image $\phi(NX)$ is a $\mathfrak{g}$-invariant open subset of $M$;
  \item  $\phi\vert_X =id_X$;
    \item and $T_x \phi = id_{N_xX}$ for all $x \in X$.
  \end{itemize}
   Here we identify $X$ with the zero section of $NX$ and the normal bundle of $X$ in $NX$ with $NX$.
\end{defn}

Finally, we introduce the notion of a transversal foliated principal bundle. Suppose that $P\rightarrow M$ is a foliated principal $G$-bundle over a foliated manifold $(M, \mathcal{F})$. Take a leafwise path $\sigma: [0, 1]\rightarrow M$ from $x_0$ and $y_0$, and take transversals $S$ and $S'$ through $x_0$ and $y_0$, respectively. The Bott connection induces a partial flat connection on $P$ and a parallel transport $P_{x_0}\rightarrow P_{y_0}$ that depends only on the homotopy class of the path $\sigma$ relative endpoints.

\begin{defn}
  We say that $\pi: P \to M$ is a \define{transversal} principal $G$-bundle, if for any leafwise path $\sigma$, the induced parallel transport depends only on the holonomy class of $\sigma$. We say a vector bundle $E\rightarrow M$ is transversal if its associated principal $GL(r)$-bundle is transversal.
\end{defn}
The normal bundle of a foliation $N\cl{F} \to M$ is transversal. In this article, we will recognize transversal principal bundles using the following example.
\begin{ex}
  \label{ex:simply-connected-leaves-transversal}
  Suppose that $\pi: P \to M$ is a foliated principal $G$-bundle over a foliation $(M,\cl{F})$ whose leaves are simply-connected. The induced parallel transport $P_{x_0} \to P_{y_0}$ depends only on the homotopy class of $\sigma$, which is entirely determined by $x_0$ and $y_0$. Therefore, the parallel transport depends only on the holonomy class of $\sigma$ too, hence $P \to M$ is transversal. 
\end{ex}

 We will need the upcoming Lemma \ref{quotient-bundle}, which relies on a refined understanding of the parallel transport induced by $\sigma$, which we outline in the following Remark \ref{rem:transversal-bundles}. 

\begin{rem}
  \label{rem:transversal-bundles}
  Suppose that $\pi: P \to M$ is a foliated principal bundle over a foliated manifold $(M, \cl{F})$, and that $\sigma$ is a leafwise path from $x_0$ to $y_0$. Fix transversals $S$ and $S'$ at $x_0$ and $y_0$ respectively. The holonomy diffeomorphism $h_\sigma^{S', S}:S \to S'$ determines a smooth map $F:[0,1] \times S \to M$ such that $F(\cdot, x)$ is a leafwise path, $F(0,x) = x$, and $F(1,x) = h_\sigma^{S',S}(x)$, for all $x \in S$. This gives rise to a family of parallel transports $\tau(x):P_x \to P_{h_\sigma^{S',S}(x)}$ for each $x \in S$. Define $\tau:P|_S \to P|_{S'}$ by $\tau|_{P_x} = \tau(x)$. Using the smooth dependence of the solution of a linear system of ODEs on parameters, one can see that $\tau$ defines a smooth isomorphism from the principal bundles $P|_S$ to $P|_{S'}$ which covers the map $h_\sigma^{S', S}:S \to S'$.
\end{rem}

\begin{lem}\label{quotient-bundle}
  Suppose that $\pi:P\rightarrow M$ is a transversal principal $G$-bundle on a foliated manifold $(M, \mathcal{F})$, and that $(M, \mathcal{F})$ is a simple foliation. Then $(P, \mathcal{F}_P)$ is also a simple foliation. Moreover, its leaf space $P/\mathcal{F}_P$ is a principal $G$-bundle over the smooth manifold $M/\mathcal{F}$.
\end{lem}
\begin{proof}
  Let $p:M \to M/\cl{F}$ be the quotient map from $M$ onto its leaf space. For a point $x \in M$, we will denote by $L_x \subset M$ the leaf of $\cl{F}$ that passes through $x$, and by $[x]$ the image of $x$ under the quotient map. For all $x_0, y_0 \in L_x$, choose a path $\sigma$ that lies in $L_{x_0}$ joining $x_0$ and $y_0$, and two transversals $S$ and $S'$ to the leaves passing through $x_0$ and $y_0$ respectively. Shrinking $S$ and $S'$ if necessary, we may assume that there is a holonomy diffeomorphism $h_\sigma^{S', S}$ from $S$ onto $S'$ that maps $x_0$ to $y_0$. Since the foliation $\cl{F}$ is simple, all leafwise paths from $x_0$ to $y_0$ have the same holonomy class. So by assumption, the Bott connection determines a unique parallel transport $\tau_{y_0, x_0}:P_{x_0} \to P_{y_0}$. For each $u \in P_{x_0}$ and $v \in P_{y_0}$, we define $u \sim v$ if and only if $v = \tau_{y_0, x_0}(u)$. This defines an equivalence relation ${\sim}$ on $P$.  We will denote by $\hat{P}$ the quotient topological space $P/{\sim}$, and by $q:P \to \hat{P}$ the quotient map. We claim that $\hat{P}$ is a principal $G$-bundle over $M/\cl{F}$ with fiber $\hat{P}_{[x]} = q(P_x)$ at $[x] \in M/\cl{F}$.

  To see this, first note that there is a natural quotient map $\hat{\pi}:\hat{P} \to M/\cl{F}$ such that the following diagram commutes.
  \begin{equation*}
    \begin{tikzcd}
      P \ar[r, "q"] \ar[d, "\pi"] & \hat{P} \ar[d, "\hat{\pi}"]\\
      M \ar[r, "p"] & M/\cl{F}.
    \end{tikzcd}
  \end{equation*}

  Next, choose an open cover $\{V_\alpha\}$ of $M$ with the following properties. For each $\alpha$, there is a smooth trivialization map $\Gamma_\alpha:\pi^{-1}(V_\alpha) \to V_\alpha \times G$ such that the following diagram commutes.
  \begin{equation*}
    \begin{tikzcd}
      \pi^{-1}(V_\alpha) \ar[r, "\Phi_\alpha"] \ar[dr, "\pi"] & V_\alpha \times G \ar[d, "\text{projection}"] \\
      & V_\alpha.
    \end{tikzcd}
  \end{equation*}
  We may also assume that each $V_\alpha$ is small enough such that the local quotient map $p:V_\alpha \to U_\alpha := p(V_\alpha) \subseteq M$ has a section $s_\alpha:U_\alpha \to V_\alpha$.

  For each $\alpha$, define $\phi_\alpha:\hat{\pi}^{-1}(U_\alpha) \to U_\alpha \times G$ by the formula $\phi_\alpha^{-1}([z],g) = q \circ \Phi_\alpha^{-1}(s_\alpha([z]), g)$, for each $[z] \in U_\alpha$ and $g \in G$. It is straightforward to check that $\phi_\alpha$ is a topological homeomorphism for each $\alpha$, and that the following diagram commutes.
  \begin{equation}
    \label{eq:trivialization}
    \begin{tikzcd}
      \hat{\pi}^{-1}(U_\alpha) \ar[r, "\phi_\alpha"] \ar[dr, "{\hat{\pi}}"] & U_\alpha \times G \ar[d, "\text{projection}"] \\
      & U_\alpha.
    \end{tikzcd}
  \end{equation}

  To finish the proof, it suffices to show that when $U_\alpha \cap U_\beta \neq \emptyset$, the transition map
  \begin{equation}
    \label{eq:transition-function}
    \phi_\alpha \circ \phi_\beta^{-1}:(U_\alpha \cap U_\beta) \times G \to (U_\alpha \cap U_\beta) \times G
  \end{equation}
  is smooth. In view of the commutativity of the diagram \eqref{eq:trivialization}, the only thing that requires proof is that the second component of the map \eqref{eq:transition-function} is smooth. But this follows easily from Remark \ref{rem:transversal-bundles}.
\end{proof}

\section{Review of Riemannian foliations}
\label{sec:review-riemannian-fol}
In this section, we review Riemannian foliations. Our main source is \cite{Mo88}.
\begin{defn}\label{bundle-like}  Let $(M, \mathcal{F})$ be a foliated manifold. We say $(M, \mathcal{F})$ is \define{Riemannian} if there exists a Riemannian metric $g_{TM}$ on $M$, called a \define{bundle-like} metric, such that $g_{TM}(X, Y)$ is a basic function for any two foliate vector fields $X$ and $Y$ that are also perpendicular to the leaves.  We say a Riemannian foliation $(M, \mathcal{F})$ is \define{complete} if it admits a complete bundle-like metric (cf.\ \cite{ACF22}).
\end{defn}

 In Example \ref{ex:homogeneous-foliations}, the foliation $\cl{F}$ is Riemannian. The metric $g$ is a compatible bundle-like metric for $\cl{F}$.

\begin{defn}\label{transverse-riemann}
  Let $(M, \mathcal{F})$ be a foliate manifold. We say that a non-negative and symmetric type $(0, 2)$ tensor $g$ is a \define{transversely Riemannian metric} on $(M, \mathcal{F})$, if it satisfies
  \begin{equation*}
   \iota_X g=0 \text{ and } \mathcal{L}_Xg=0 \text{ for all } X\in\mathfrak{X}(\mathcal{F}). 
 \end{equation*}
Every bundle-like metric $g_{TM}$ induces a transversely Riemannian metric $g$, and every transversely Riemannian metric $g$ is induced by some bundle-like metric $g_{TM}$. Thus, to specify a Riemannian foliation $(M, \cl{F})$, it is equivalent to give a bundle-like metric, or a transversely Riemannian metric.
 
\end{defn}

Let $g$ be a transversely Riemannian metric on a foliated manifold $(M, \mathcal{F})$, and let $\overline{X} \in \mathfrak{X}(M, \mathcal{F})$ be a transverse vector field. Choose a foliate vector field representing $\overline{X}$ and define
\begin{equation*}
  \mathcal{L}_{\overline{X}}g:=\mathcal{L}_Xg.
\end{equation*}
It follows from Definition \ref{transverse-riemann} that the definition does not depend on the choice of a particular foliated vector field representing $\overline{X}$. We say the transverse vector field $\overline{X}$ is \define{transverse Killing} if $\mathcal{L}_{\overline{X}}g=0$. From the usual Cartan identity, we see that the space of all transverse Killing vector fields on $(M, \mathcal{F}, g)$ forms a Lie subalgebra of $\mathfrak{X}(M, \mathcal{F})$, which we denote by $\text{Iso}(M, \mathcal{F}, g)$. We say a transverse Lie algebra action $a: \mathfrak{g} \to \mathfrak{X}(M, \mathcal{F})$ is \define{isometric}, if $a(\xi)\in  \text{Iso}(M, \mathcal{F}, g)$. 

\begin{defn}
  Let $(M, \mathcal{F})$ be a foliated manifold of co-dimension $q$. We call a set of $q$ foliate vector fields $\{X_1, \cdots, X_q\}$ a \define{transverse parallelism} if they are everywhere linear independent. We say that $(M, \mathcal{F})$ is \define{transversely parallizable} if it admits a transverse parallelism.
\end{defn}

\begin{ex}\label{ex:transverse-parallelism-metric}
  Let $(M, \mathcal{F})$ be a transversely parallizable foliation of co-dimension $q$. Then $ \mathcal{F}$ must be a Riemannian foliation. Indeed, let $\{X_1, \cdots, X_q\}$ be a transverse parallelism on $(M, \mathcal{F})$. Then the equations
  \begin{align*}
    \alpha^i(X_j) &=\delta_i^j \text{ for } 1\leq i, j\leq q \\
    \alpha^i(X) &=0 \text{ for }  X\in \mathfrak{X}(M, \mathcal{F})
  \end{align*}
  uniquely determine $q$ many basic 1-forms $\alpha^1, \cdots, \alpha^q$. The metric
  \begin{equation*}
    g := \sum_{i=1}^q \alpha^i \otimes \alpha^i
  \end{equation*}
  is a transversely Riemannian metric.
\end{ex}




\subsection{Molino's structure theory of a Riemannian foliation}\label{Molino-theory}

Let $g$ be a transverse Riemannian metric on a complete Riemannian foliation $(M, \mathcal{F})$. Let $\pi: P\rightarrow M$ be the transverse orthonormal frame bundle associated to the transverse Riemannian metric $g$, whose fiber over $x \in M$ consists of orthogonal transformations $\R^q \to N_x\cl{F}$. Let $K := O(q)$ be the structure Lie group of $P$.  Then the foliation $\mathcal{F}$ naturally lifts to a transversely parallizable foliation $\mathcal{F}_P$ on $P$ that is invariant under the action of $K$; moreover, there is a locally trivial fibration $\rho: P\rightarrow W$, called the basic fibration, whose fibers are leaf closures of $\mathcal{F}_P$. In a diagram,
\begin{equation*}
  \begin{tikzcd}
    & P \ar[dl, "\pi"'] \ar[dr, "\rho"] & \\
    M & & W = P/\ol{\cl{F}_P}.
  \end{tikzcd}
\end{equation*}
Let $\ol{L_P} \subset P$ be the closure of a leaf $L_P$ of $(P, \mathcal{F}_P)$. Then $(\overline{L_P}, \mathcal{F}_P\vert_{\overline{L_P}})$ is a complete transversely parallizable foliation, for which the Lie algebra of transverse vector fields $\mathfrak{g}:=\mathfrak{X}(\overline{L_P}, \mathcal{F}_P\vert_{\overline{L_P}})$ is finite dimensional. Indeed, $(\overline{L_P}, \mathcal{F}_P\vert_{\overline{L_P}})$ is a complete $\mathfrak{g}$-Lie foliation, in the sense of Definition \ref{Lie-foliation}. Since $(P, \mathcal{F}_P)$ is a complete transversely parallizable foliation, the Lie algebra $\mathfrak{g}$ is independent of the choice of a particular leaf closure. It is also shown in \cite[Section 5.1]{Mo88} that $\mathfrak{g}$ is independent of the choice of a particular transverse Riemannian metric on $(M, \mathcal{F})$. We call $\mathfrak{g}$ the \textbf{structure Lie algebra} of $(M, \mathcal{F})$, and will denote it by $\mathfrak{g}=\mathfrak{g}(M, \mathcal{F})$.  Also note that the leaf $L_P$ projects by $\pi$ to a leaf $L$ in $(M, \mathcal{F})$. We will need the following description of the closure of $L$.
 
 \begin{lem}[{\cite[Lemma 5.1]{Mo88}}] \label{leaf-closure}
   The closure $\overline{L_P}$ of $L_P$ projects by $\pi$ onto the closure of $L$. Moreover, the map $\pi\vert_{\overline{L_P}}: \overline{L_P}\rightarrow \overline{L}$ has constant rank, and its image $\overline{L}$ is an embedded foliated submanifold of $M$.
 \end{lem}
 
 

We may give an explicit transverse parallelism of $(P, \cl{F}_P)$. By definition, a transverse orthonormal frame $z \in P$ over a point $x$ of $M$ is an orthonormal basis $\{\overline{X}_1, \ldots, \overline{X}_q\}$ in the normal space $N_x\mathcal{F}=T_xM/T_x\mathcal{F}$. We will identify $z$ with a linear isometry $z: (\mathbb{R}^q , \langle\,,\,\rangle)\rightarrow (N_x\mathcal{F}, g_x)$, where $\langle\,,\,\rangle$ is the standard Euclidean inner product on $\mathbb{R}^q$. In this notation, the fundamental 1-form $\omega$ of $P$ is a $\mathbb{R}^q$-valued 1-form defined by
\begin{equation*}
   \omega(X_z)=z^{-1}(\pi_*(X_z)), \quad \text{for } z\in P, \text{ and } X_z\in T_zP.
\end{equation*}
Since the foliation is Riemannian, we have a unique transverse Levi-Civita connection $1$-form $\theta_{LC}$ on $P$, which is a basic form that takes values in the Lie algebra $o(q)$. Let  $\{e_1,\cdots, e_q\}$  be an orthonormal basis of $\mathbb{R}^q$. Then the equations, for $1 \leq i \leq q$,
\begin{align*}
  \omega(u_i) &= e_i \\
  \theta(u_i) & =0
\end{align*}
determine $q$ many horizontal foliate vector fields $\{ u_1, \ldots, u_q\}$ on $(P, \mathcal{F}_P)$. On the other hand, choosing a basis $\{\lambda_1, \ldots, \lambda_{\frac{q(q-1)}{2}}\}$ of $o(q)$, the Lie algebra action of $o(q)$ on $P$ gives $\frac{q(q-1)}{2}$ vertical foliate vector fields $\{\lambda_{1,P}, \ldots, \lambda_{\frac{q(q-1)}{2}, P} \}$ on $P$ that are vertical. Putting them together,
\begin{equation}\label{t-parallelism}
  \{u_1, \ldots, u_q,  \lambda_{1,P}, \ldots, \lambda_{\frac{q(q-1)}{2},P} \}
\end{equation}
provides a transverse parallelism on $P$. As a consequence, we may construct a transverse Riemannian metric $g_P$ on $P$ following Example \ref{ex:transverse-parallelism-metric}. This transverse parallelism is complete, see \cite[Page 17]{ACF22}.


Let $X$ be a foliate vector field on $(M, \mathcal{F})$. Note that by \cite[Lemma 3.4]{Mo88}, $X$ naturally lifts to a foliate vector $X_P$ on $P$ that projects to $X$, and which has the property that $\mathcal{L}_{X_P} \theta_{LC}=0$. If $X$ is also Killing with respect to $g$, then by \cite[Proposition 3.4]{Mo88}, $X_P$ commutes with the transverse parallelism on $P$ given in the previous paragraph. So $X_P$ is Killing with respect to $g_P$. When $X$ is tangent to the leaves of $\mathcal{F}$, it is easy to check that $X_P$ is tangent to the leaves of the lifted foliation $\mathcal{F}_P$. Therefore for any open subset $V \subseteq M$, this lifting gives a Lie algebra homomorphism of transverse Killing vector fields:
\begin{equation}\label{lifting-homo}
  \pi^{\sharp}: \Iso(V, \mathcal{F}\vert_V, g\vert_V)\rightarrow \Iso(\pi^{-1}(V), \mathcal{F}_P\vert_{\pi^{-1}(V)}, g_P\vert_{\pi^{-1}(V)}).
\end{equation}
 
For each open $U \subseteq P$, define
\begin{equation*}
 \mathbf{C}(U, \mathcal{F}_P\vert_{U}) := \{ \overline{X}\in \mathfrak{X}(U, \mathcal{F}_P|_{U}) \mid  [\overline{X}, \overline{Y}]=0 \text{ for all } \overline{Y}\in \mathfrak{X}(P, \mathcal{F}_P)\}.
\end{equation*}
The assignment $U\mapsto \mathbf{C}(U, \mathcal{F}_P\vert_U)$ defines a pre-sheaf.  The associated sheaf is called the Molino sheaf on $(P, \mathcal{F}_P)$. Now, for each open $V \subseteq M$, define
\begin{equation*}
  \mathbf{C}(V, \mathcal{F}|_V) := \{\overline{X}\in \Iso(V, \mathcal{F}\vert_V, g\vert_V)\} \mid \pi^{\sharp} \overline{X}\in \mathbf{C}(\pi^{-1}(V), \mathcal{F}_P|_{\pi^{-1}(V)})\}
\end{equation*}
The  assignment
$V\mapsto \mathbf{C}(V, \mathcal{F}\vert_V)$ is a pre-sheaf, and the associated sheaf, denoted by $\mathbf{C}(M,\mathcal{F})$,  is called the Molino sheaf on $(M, \mathcal{F})$. It follows from \cite[Proposition 3.4]{Mo88} that the lifting homomorphism $\pi^{\sharp}$ induces an isomorphism $\mathbf{C}(M, \mathcal{F})\cong \left(\pi_*\mathbf{C}(P, \mathcal{F}_P)\right)^K$. We now summarize Molino's structure theory.

\begin{thm}[{\cite[Theorem 5.2]{Mo88}}]\label{Molino-sheaf}\ 
  \begin{enumeratea}
  \item The Molino sheaves $\mathbf{C}(P, \mathcal{F}_P)$ and $\mathbf{C}(M, \mathcal{F})$ are both locally constant. The orbits of $\mathbf{C}(P, \mathcal{F}_P)$ and $\mathbf{C}(M, \mathcal{F})$ are the leaf closures of $(P, \mathcal{F}_P)$ and $(M, \mathcal{F})$ respectively.  
\item The typical fiber of $\mathbf{C}(M, \mathcal{F})$ is the Lie algebra $\mathfrak{g}^-$ opposite to the structure Lie algebra $\mathfrak{g}(M, \mathcal{F})$.  Furthermore, all the global transverse vector fields commute with $\mathbf{C}(M, \mathcal{F})$. As an immediate result, when $\mathbf{C}(M, \mathcal{F})$ is globally constant, the structure Lie algebra $\mathfrak{g}(M, \mathcal{F})$ is abelian.
 \end{enumeratea}
\end{thm}

We use Molino's sheaf to define Killing Riemannian foliations.
\begin{defn} \label{Killing}We say a Riemannian foliation $(M, \mathcal{F})$ is \textbf{Killing}, if its Molino sheaf $\mathbf{C}(M, \mathcal{F})$ is globally constant.
\end{defn}
\begin{ex}
  \label{ex:killing-foliations}
  Here are some examples of Killing foliations.
  \begin{itemize}
  \item If $M$ is simply-connected, then any Riemannian foliation $(M, \mathcal{F})$ must be a Killing foliation. This is also immediate from the definition of Killing foliation; see \cite[Proposition 5.5]{Mo88}. Its structure algebra is trivial.
  \item Recall the setting of Example \ref{ex:homogeneous-foliations}: $(M, g)$ is a compact Riemannian manifold, and the leaves of $\cl{F}$ are the orbits of a connected group of isometries. Then $(M, \cl{F})$ is a Killing foliation. Its structure algebra is $\fk{g} := \Lie(K)/\Lie(H)$, where $K = \overline{H}$.

    Specializing this example, we may take $(M, \alpha)$ to be a compact contact manifold, and impose the following compatibility conditions on the Riemannian metric $g$. Let $R$ be the Reeb vector field of associated to $\alpha$, defined by $\alpha(R) = 1$ and $\iota_R d\alpha = 0$. We require
      \begin{enumerate}[label=(\roman*)]
      \item the flow of $\R$ preserves the metric $g$, and
      \item $\ker \alpha \subseteq TM$ admits an almost complex structure $J$ such that
        \begin{equation*}
          g(X,Y) = d\alpha(X, JY) \text{ for all } X,Y \in \Gamma(\ker \alpha).
        \end{equation*}
      \end{enumerate}
      Under these conditions, we call $(M, \alpha, g)$ a \define{K-contact manifold}. These include \define{Sasakian manifolds}. We refer to \cite{BG08} for definitions and more details. The flow of the Reeb vector field $R$ induces a Killing foliation on $M$.
  \end{itemize}
\end{ex}

    Let $(M, \mathcal{F})$ be a complete Killing foliation. Then its structure Lie algebra $\mathfrak{g}$ is abelian; moreover, there is a natural transverse isometric $\mathfrak{g}$-action on $(M, \mathcal{F})$. The argument used in \cite[Proposition 3.3.4]{LS20} can be adapted to prove the following result.

\begin{prop}\label{tubular} Let $(M, \mathcal{F})$ be a complete Killing foliation, and let $X$ be a foliated closed submanifold of $M$.  Then $X$ admits a $\mathfrak{g}$-invariant tubular neighbourhood $\phi: NX \hookrightarrow M$.
\end{prop}

\subsection{Molino's sheaf via the holonomy pseudogroup}
\label{sec:molin-sheaf-pseud}

In this subsection, we review an alternative approach to Molino's sheaf through the holonomy pseudogroup $\Psi(M, \cl{F})$. It is relevant for the discussion in Subsection \ref{sec:comp-haefl-model}, and may be skipped on first reading. Let $\{\varphi_\alpha:U_\alpha \to \R^p \times \R^q\}$ be a foliated atlas of $(M, \cl{F})$. Set $s_\alpha := \pr_2 \circ \varphi_\alpha$. Then we have diffeomorphisms (denoting $U_{\alpha\beta} := U_\alpha \cap U_\beta$)
\begin{equation*}
  h_{\beta\alpha}:s_\alpha(U_{\alpha\beta}) \to s_\beta(U_{\alpha\beta})
\end{equation*}
such that $h_{\beta\alpha} \circ s_\alpha|_{U_{\alpha\beta}} = s_\beta|_{U_{\alpha\beta}}$. The data $\{(s_\alpha, U_\alpha, h_{\beta\alpha})\}$ is a Haefliger cocycle for $(M, \cl{F})$. Let $S_\alpha := s_\alpha(U_\alpha)$, and view $S := \bigsqcup_\alpha S_\alpha$ as a complete transversal to $\cl{F}$. A transverse metric $g$ for $(M, \cl{F})$ induces a Riemannian metric on $S$. The holonomy pseudogroup $\Psi_S(M,\cl{F})$ is generated by the $h_{\beta\alpha}$, and consists of local isometries of $S$.

\begin{defn}
  Let $\mathbf{C}(S, \Psi_S(M, \cl{F}))$ be the sheaf over $S$ defined as follows. Assign to each open $V \subseteq S$ the Lie algebra consisting of vector fields $\xi$ on $V$ such that: for each $y \in V$, there is a neighbourhood $V_y$ of $y$, and $\epsilon > 0$, such that $\exp(t\xi)$ is defined on all of $V_y$ whenever $|t| < \epsilon$, and $\exp(t\xi) \in \ol{\Psi_S(M, \cl{F})}$.
\end{defn}

For the closure of a pseudogroup, see \cite{H88}. We now work to show that $\mathbf{C}(M, \cl{F})$ is, in the appropriate sense, a pullback of $\mathbf{C}(S, \Psi_S(M, \cl{F}))$ to $M$.  Let $\mbf{C}_\alpha$ be the inverse image of the sheaf $\mbf{C}(S, \Psi_S(M, \cl{F}))|_{S_\alpha}$ along $s_\alpha$. This is a sheaf on $U_\alpha$, given explicitly by assigning to each $U \subseteq U_\alpha$ the set $\mbf{C}(S, \Psi_S(M, \cl{F}))(s_\alpha(U))$ (because $s_\alpha$ is an open map). Moreover, for each $\alpha, \beta$, we have a family of isomorphisms of sheaves
 \begin{equation*}
   \eta_{\beta\alpha}: \mbf{C}_\alpha|_{U_{\alpha\beta}} \to \mbf{C}_\beta|_{U_{\alpha \beta}}
 \end{equation*}
 satisfying the cocycle condition. Indeed, given $U \subseteq U_{\alpha \beta}$, we define $(\eta_{\beta \alpha})_U: \mbf{C}_\alpha(U) \to \mbf{C}_\beta(U)$ to be the map $\xi \mapsto (h_{\beta \alpha})_*\xi$. To check this has the desired properties, first note that if $\xi$ is defined on $s_\alpha(U)$, then since $h_{\beta \alpha}$ restricts to a diffeomorphism $s_\alpha(U) \to s_\beta(U)$, we have that $(h_{\beta \alpha})_*\xi$ is defined on $s_\beta(U)$, and we have an inverse given by $(h_{\alpha \beta})_*$. A cocycle condition for $(\eta_{\beta \alpha})$ follows from that for the Haefliger cocycle. We may therefore glue the sheaves $\mbf{C}_\alpha$ along the cocycle $(\eta_{\beta \alpha})$, to obtain a sheaf $\mbf{C}$ on $M$. Explicitly, the sheaf $\mbf{C}$ assigns to each open $U \subseteq M$ the set
 \begin{equation*}
   \{(\xi_\alpha) \in \bigsqcup_\alpha \mbf{C}_\alpha(U \cap U_\alpha) \mid (h_{\beta \alpha})_*\xi_\alpha|_{s_\alpha(U \cap U_{\alpha \beta})} = \xi_\beta|_{s_\beta(U \cap U_{\alpha \beta})}\}. 
 \end{equation*} 

 \begin{lem}\label{lem:3}
   The sheaves $\mbf{C}(M,\cl{F})$ and $\mbf{C}$ are isomorphic.
 \end{lem}
 \begin{proof}
   The isomorphism $\mbf{I}:\mbf{C}(M, \cl{F}) \to \mbf{C}$ is defined as follows: for each open $U \subseteq M$, take
   \begin{equation*}
     \mbf{I}_U :\mbf{C}(M, \cl{F})(U) \to \mbf{C}(U), \quad X \mapsto ((s_\alpha)_*(X|_{U \cap U_\alpha})).
   \end{equation*}
 \end{proof}
Since $\mbf{C}$ is locally given by pullbacks of $\mbf{C}(S, \Psi_S(M, \cl{F}))$, Lemma \ref{lem:3} implies that the stalk $\mbf{C}(M, \cl{F})_x$ is isomorphic to the stalk $\mbf{C}(S, \Psi_S(M, \cl{F}))_x$. There is an action of $(\ol{\Psi_S(M, \cl{F})})_x$ (the isotropy group at $x$ of the closure) on the stalk $\mbf{C}(S, \Psi_S(M, \cl{F}))_x$, given by:
 \begin{equation}\label{eq:action-on-stalk}
   h \cdot \germ_x \xi := \text{ the germ at }x \text{ of } \frac{d}{dt} \Big|_{t=0} h \circ \exp(t\xi) \circ h^{-1}.
 \end{equation}
Therefore, the isotropy group also acts on the stalk $\mbf{C}(M, \cl{F})_x$.

\section{The leaf space of a developable foliation}
\label{sec:leaf-space-devel}
One of the key observations we will make in our study of Killing foliations in the next section is that they are locally developable. Here we gather some results on the transverse geometry of developable foliations. We will end by introducing complete Lie-$\fk{g}$ foliations and Fedida's theorem, which states that such foliations are developable. We take the definition of developable from \cite[Page 161]{MM03}. 
\begin{defn}
  \label{def:developable}
  A foliation $(X, \cl{F})$ is \define{developable} if, for some cover $p:\hat{X} \to X$ of $X$, the pullback foliation $\hat{\cl{F}} := p^{-1}(\cl{F})$ is simple. Recall this means that the leaf space $M := \hat{X}/\hat{\cl{F}}$ is a manifold, and the quotient map $q:\hat{X} \to M$ is a surjective submersion whose fibers are the leaves of $\hat{\cl{F}}$. We call $q$ a \define{developing map} for $\cl{F}$. 
\end{defn}
Since to be developable depends on the existence of the cover $p$, we always equip developable foliations with the cover $p$ satisfying Definition \ref{def:developable}. We will denote the group of deck transformations of $p$ by $\Gamma$, and take $\Gamma$ to act on $\hat{X}$ from the left.

\begin{lem}
     \label{lem:gamma-acts}
     Let $(X, \cl{F}, p)$ be a developable foliation. The group $\Gamma$ acts smoothly on the leaf space $M$, and the developing map $q:\hat{X} \to M$ is $\Gamma$-equivariant.
   \end{lem}
   \begin{proof}
    This is a consequence of the fact that the $\Gamma$ action on $\hat{X}$ preserves the fibers of $p$.
  \end{proof}

  We will describe how the action of $\Gamma$ on $M$ captures the transverse geometry of $(X, \cl{F})$. We take two approaches, first through diffeology, then through Lie groupoids. Diffeologically, we can compare the leaf space $X/\cl{F}$ to the quotient space $M/\Gamma$ directly.

   \begin{prop}
  \label{prop:develop-leaf-space}
  If $(X, \cl{F}, p)$ is a developable foliation, and $p$ is a Galois cover, then there is a unique diffeological diffeomorphism $\Phi:X/\cl{F} \to M/\Gamma$ making the following commute:
  \begin{equation}\label{eq:2}
    \begin{tikzcd}
      & \hat{X} \ar[dl, "p"'] \ar[dr, "q"] & \\
      X \ar[d, "\pi_{\cl{F}}"'] & & M \ar[d, "\pi_{\Gamma}"] \\
      X/\cl{F} \ar[rr, "\Phi"] & & M/\Gamma.
    \end{tikzcd}
  \end{equation}
\end{prop}
We will use $L$ to denote a leaf of $\cl{F}$, use $L_x$ to denote the leaf through $x$, and use $[m]$ to denote the $\Gamma$-orbit of $m \in M$.
\begin{proof}
    Both $\pi_{\cl{F}} \circ p$ and $\pi_\Gamma \circ q$ are subductions (the diffeological version of submersion, see \cite{IZ13}), so any $\Phi$ completing \eqref{eq:2} is smooth, and if it is invertible, it is a diffeomorphism. Uniqueness follows from surjectivity of $\pi_{\cl{F}} \circ p$. Necessarily, we must have
  \begin{equation*}
    \Phi(L_{p(\hat{x})}) := [q(\hat{x})].
  \end{equation*}

  To show $\Phi$ is well-defined, suppose that $\hat{x}, \hat{y} \in \hat{X}$ are such that $x:= p(\hat{x})$ and $y := p(\hat{y})$ are in the same leaf $L$ of $\cl{F}$. We require that $q(\hat{x})$ and $q(\hat{y})$ are in the same $\Gamma$-orbit. Choose a path $\lambda$ from $x$ to $y$ that stays in $L$. Because $p:\hat{X} \to X$ is a cover, $\lambda$ lifts to a unique path $\hat{\lambda}$ starting at $\hat{x}$ and ending at some $\hat{z}$, with $p(\hat{z}) = y$. The path $\hat{\lambda}$ is contained in a connected component of $p^{-1}(L)$, which is precisely a leaf of $\hat{\cl{F}}$. Therefore $\hat{x}$ and $\hat{z}$ are in the same leaf of $\hat{\cl{F}}$. But the leaves of $\hat{\cl{F}}$ are also the fibers of $q$, so $q(\hat{x}) = q(\hat{z})$. On the other hand, $p:\hat{X} \to X$ is a Galois cover with deck transformation group $\Gamma$, and $p(\hat{y}) = p(\hat{z}) = y$, so there is some $\gamma \in \Gamma$ with $\gamma \cdot \hat{z} = \hat{y}$. The $\Gamma$-equivariance of $q$ implies
  \begin{equation*}
    \gamma \cdot q(\hat{x}) = \gamma \cdot q(\hat{z}) = q(\gamma \cdot \hat{z}) = q(\hat{y}),
  \end{equation*}
  as desired.

  To see $\Phi$ is invertible, observe that its inverse must be given by: for $\hat{x} \in \hat{X}$,
  \begin{equation*}
    \Phi^{-1}([q(\hat{x})]) = L_{p(\hat{x})}.
  \end{equation*}
  To see this is well-defined, suppose that $\hat{x}, \hat{y} \in \hat{X}$ are such that $\gamma \cdot q(\hat{x}) = q(\hat{y})$ for some $\gamma \in \Gamma$. We must show $p(\hat{x})$ and $p(\hat{y})$ are in the same leaf. By $\Gamma$-equivariance of $q$, we find that $\gamma \cdot \hat{x}$ and $\hat{y}$ are in the same fiber of $q$, or equivalently, the same leaf of $\hat{\cl{F}}$. Therefore their images under $p$ are in the same leaf of $\cl{F}$. But $p$ is $\Gamma$-invariant, so $p(\gamma \cdot \hat{x}) = p(\hat{x})$, and this observation completes the proof. 
\end{proof}

Before describing the Lie groupoid approach to the transverse geometry of $(X, \cl{F}, p)$, we briefly review Lie groupoids. 

  \subsection{Review of Lie groupoids}
\label{sec:review-lie-groupoids}

A \define{Lie groupoid} is a category $G \rra G_0$, which we may denote by $G$, satisfying two conditions:
\begin{itemize}
\item (Lie) the arrow space $G$ is a not necessarily Hausdorff nor second-countable manifold, the base $G_0$ is a manifold, all the structure maps are smooth, and moreover the source $s$ and target $t$ are surjective submersions with Hausdorff fibers;
\item (groupoid) every arrow has an inverse.
\end{itemize}
Given a Lie groupoid $G \rra G_0$, the \define{orbit} $\cl{O}_x$ of $G$ through $x \in G_0$ is the subset $\cl{O}_x := t(s^{-1}(x))$. Each orbit is a weakly-embedded submanifold\footnote{A set, equipped with the subset diffeology, that is locally diffeomorphic to $\R^n$ and whose inclusion map is an immersion. The smooth structure on $\cl{O}_x$ is therefore entirely determined by $M$.} of $G_0$. The group $G_x := s^{-1}(x) \cap t^{-1}(x)$ is the \define{isotropy group} at $x$, and it is naturally a Lie group.

If $\dim G = \dim G_0$, we call the Lie groupoid $G$ \define{\'{e}tale}. \'{E}tale groupoids are intimately related to pseudogroups. Given a pseudogroup $\Psi$ on $M$, we may form the \'{e}tale germ groupoid, defined in Example \ref{ex:three-groupoids} below. Conversely, to an \'{e}tale groupoid $G$, we associate the pseudogroup $\Psi(G)$ on $G_0$ generated by transitions of the form $t \circ \sigma$, where $\sigma$ is a local inverse of $s$.

We will encounter the following Lie groupoids.
\begin{ex}\
  \label{ex:three-groupoids}
  \begin{itemize}
  \item  Given an action of a Lie group $\Gamma$ on $M$, we form the \define{action groupoid} $\Gamma \ltimes M$. Its base space is $M$ and its arrow space is $\Gamma \times M$. A pair $(\gamma, m)$ is an arrow $m \xar{(\gamma, m)} \gamma \cdot m$. Multiplication is given by $(\gamma', m') \cdot (\gamma, m) = (\gamma' \gamma, m)$. If $\Gamma$ is a discrete Lie group, then $\Gamma\ltimes M$ is \'{e}tale. The orbits of an action groupoid are the orbits of the action.

    \item Given a foliation $(M, \cl{F})$, we form the \define{holonomy groupoid} $\Hol(\cl{F})$. This has base space $M$ and arrow space
    \begin{equation*}
      \Hol(\cl{F}) := \{\text{holonomy classes of } \lambda \mid \lambda \text{ is a leafwise path}\}.
    \end{equation*}
    Each holonomy class $[\lambda]$ is an arrow $\lambda(0) \xar{[\lambda]} \lambda(1)$, and multiplication is given by concatenation. The set $\Hol(\cl{F})$ admits a manifold structure (see \cite[Proposition 5.6]{MM03}), but it may not be Hausdorff. The orbits of a holonomy groupoid $\Hol(\cl{F})$ are the leaves of $\cl{F}$.

    If $S$ is a complete transversal to $\cl{F}$, we can form the groupoid $\Hol_S(M, \cl{F}) \rra S$, whose arrow space consists of arrows in $\Hol(\cl{F})$ with source and target in $S$. This is an example of a pullback groupoid (see Example \ref{ex:pullback-groupoid}). It is an \'{e}tale Lie groupoid, and $\Psi(\Hol_S(\cl{F}))$ is the holonomy pseudogroup $\Psi_S(M, \cl{F})$.

  \item Given a pseudogroup $\Psi$ on a manifold $M$, we form the \define{germ groupoid} $\Gamma(\Psi)$. This has base space $M$ and arrow space
    \begin{equation*}
      \Gamma(\Psi) := \{\germ_x\psi \mid x \in \dom \psi \text{ and } \psi \in \Psi\}.
    \end{equation*}
    The arrow $\germ_x\psi$ is an arrow $x \xar{\germ_x\psi} \psi(x)$. The smooth structure on $\Gamma(\Psi)$ is given by the charts $x \mapsto \germ_x\psi$. The arrow space $\Gamma(\Psi)$ need not be Hausdorff nor second-countable. This is an \'{e}tale Lie groupoid.
  \end{itemize}
\end{ex}

The transverse geometry of a Lie groupoid is encoded in its Morita equivalence class. There are many ways to describe Morita equivalences. We choose the one from \cite{HF2019}.

\begin{defn}
  A \define{Morita map} $\phi:G \to H$ is a smooth functor such that:
  \begin{itemize}
  \item it descends to a homeomorphism $\overline{\phi}:G_0/G \to H_0/H$;
  \item the induced homomorphism $\phi_x:G_x \to H_{\phi(x)}$ is an isomorphism for all $x$;
    \item the normal derivative $N_x\phi:N_x\cl{O}_x \to N_{\phi(x)}\cl{O}_{\phi(x)}$ is an isomorphism.
    \end{itemize}
    Two Lie groupoids $G$ and $H$ are \define{Morita equivalent} if there is a third Lie groupoid $K$, and Morita maps $K \to G$ and $K \to H$. 
  \end{defn}
This notion of Morita equivalence agrees with the others found in the literature. For a proof, see \cite[Proposition 6.1.1]{HF2019}.
  
\begin{ex}
  \label{ex:pullback-groupoid}
  Given a Lie groupoid $G \rra G_0$ and a map $p:M \to G_0$, we may form the pullback (not necessarily Lie) groupoid $p^*G$, whose arrow space consists of triples $(y,g,x)$ where $g \in G$ is an arrow from $p(x)$ to $p(y)$. Whenever the map
  \begin{equation*}
    t \circ \pr_1 : G \fiber{s}{p} M \to G_0, \quad (g, m) \mapsto t(g)
  \end{equation*}
  is a submersion, the groupoid $p^*G$ is a Lie groupoid. Whenever this map is additionally surjective (e.g.\ if $p$ is a surjective submersion), the natural functor $p: p^*G \to G$,
  \begin{equation*}
    m \mapsto p(m), \quad (y,g,x) \mapsto g
  \end{equation*}
  is a Morita map.
\end{ex}

\subsection{The Morita class of the holonomy groupoid }
\label{sec:morita-class-holon}

Given a developable foliation $(X, \cl{F}, p)$, we will describe the Morita equivalence class of its holonomy groupoid $\Hol(\cl{F})$ in terms of the $\Gamma$ action on $M$.
Let $h:\Gamma \to \Diff(M)$ denote the action homomorphism. Let $G \rra M$ be the germ groupoid of the pseudogroup generated by $h(\Gamma) \leq \Diff(M)$. Arrows of $G$ are of the form $\germ_m h(\gamma)$.

\begin{prop}
  \label{prop:morita-to-pseudogroup}
  If $(X, \cl{F}, p)$ is a developable foliation, and $p$ is a Galois cover, then the holonomy groupoid $\Hol(\cl{F})$ is Morita equivalent to $G$, the germ groupoid of the pseudogroup generated by the $\Gamma$ action on $M$.
\end{prop}

\begin{proof}
We seek the diagram below:
  \begin{equation*}
    \begin{tikzcd}
      \Hol(\cl{F}) \ar[d, shift right] \ar[d, shift left] & \ar[l, "p_1"'] p^*(\Hol(\cl{F})) \ar[d, shift right] \ar[d, shift left] \ar[r, "q_1"] & G \ar[d, shift right] \ar[d, shift left] \\
      X & \ar[l, "p"'] \hat{X} \ar[r, "q"] & M,
    \end{tikzcd}
  \end{equation*}
where $(p,p_1)$ and $(q,q_1)$ are Morita maps.
  
Here $p^*(\Hol(\cl{F}))$ is the pullback of $\Hol(\cl{F})$ along $p$, and $(p, p_1)$ is the natural functor, as described in Example \ref{ex:pullback-groupoid}.  Because $p$ is a surjective submersion, the natural functor is a Morita map. Recall that arrows in $p^*(\Hol(\cl{F}))$ are of the form $\hat{x} \xar{(\hat{y}, [\lambda], \hat{x})} \hat{y}$, where $\lambda$ is a leafwise path from $p(\hat{x})$ to $p(\hat{y})$.

We define $q_1$ as follows. First, note that the map
\begin{equation*}
  \Gamma \times \hat{X} \to \hat{X} \times_p \hat{X}, \quad (\gamma, \hat{x}) \mapsto (\gamma \cdot \hat{x}, \hat{x})
\end{equation*}
is a diffeomorphism, because $p:\hat{X} \to X$ is Galois, hence a principal (left) $\Gamma$-bundle. Denote its inverse by
\begin{equation*}
  (\hat{x},\hat{y}) \mapsto (d(\hat{x}, \hat{y}), \hat{y}),
\end{equation*}
where $d:\hat{X} \times_p \hat{X} \to \Gamma$ is a smooth map. Note that $d(\hat{x}, \hat{y})^{-1} = d(\hat{y}, \hat{x})$, and that $d(\hat{x}, \hat{y})$ is defined by the relation
\begin{equation*}
  d(\hat{x}, \hat{y}) \cdot \hat{y} = \hat{x}.
\end{equation*}
Now, take an arrow $(\hat{y}, [\lambda], \hat{x})$. Take a leafwise path $\lambda$ representing $[\lambda]$, and lift it along the cover $p$ to a unique leafwise path $\hat{\lambda}$ starting at $\hat{x}$. Denote the other end of $\hat{\lambda}$ by $\hat{z}$, which must be in the same leaf of $\hat{\cl{F}}$ as $\hat{x}$. Both $\hat{z}$ and $\hat{y}$ are in the same fiber of $p$. We define
\begin{equation*}
  q_1(\hat{y}, [\lambda], \hat{x}) := \germ_{q(\hat{z})} h(d(\hat{y}, \hat{z})).
\end{equation*}

To check that $q_1$ is well-defined, let $\lambda_1 \in [\lambda]$ be another representative of $[\lambda]$, and lift it to a leafwise path $\hat{\lambda}_1$ starting at $\hat{x}$ and ending at $\hat{z}_1$. We require that both $h(d(\hat{y}, \hat{z}))$ and $h(d(\hat{y}, \hat{z}_1))$ have the same germ at $q(\hat{z})$. It suffices to show that
\begin{equation*}
\gamma := d(\hat{z}_1, \hat{y}) \cdot d(\hat{y}, \hat{z}) = d(\hat{z}_1, \hat{z})  
\end{equation*}
locally acts as the identity on $M$ in a neighbourhood of $q(\hat{z})$.
Let $S$ be a transversal to $\cl{F}$ at $y := p(\hat{y})$. The holonomy of $\lambda_0 := \lambda^{-1} * \lambda_1$ is represented by a transition $\hol^S(\lambda_0)$ of $S$ which, by our assumption that $[\lambda] = [\lambda_1]$, is a restriction of the identity. The pre-image $\hat{S} := p^{-1}(S)$ is a transversal to $\hat{\cl{F}}$, and the holonomy of the lifted path $\hat{\lambda}_0 := (\hat{\lambda})^{-1} * \hat{\lambda}_1$, which takes $\hat{z}$ to $\hat{z}_1$, is represented by a transition $\hol^{\hat{S}}(\hat{\lambda}_0)$ fitting into the diagram
\begin{equation*}
  \begin{tikzcd}[sep=large]
    (\hat{S}, \hat{z}) \ar[r, dashed, "\hol^{\hat{S}}(\hat{\lambda}_0)"] \ar[d, "p"] & (\hat{S}, \hat{z}_1) \ar[d, "p"] \\
    (S, y) \ar[r, dashed, "\hol^S(\lambda_0) = \id"] & (S, y).
  \end{tikzcd}
\end{equation*}
Commutativity implies that the map
\begin{equation*}
 \hat{S} \dashrightarrow \Gamma, \quad \hat{t} \mapsto d(\hol^{\hat{S}}(\hat{\lambda}_0)(\hat{t}), \hat{t})
\end{equation*}
is well-defined. It is smooth, being the composition of smooth functions, hence constant on connected neighbourhoods in $\hat{S}$. In particular, its value at $\hat{z}$ is $d(\hat{z}_1, \hat{z}) = \gamma$, so it is identically $\gamma$ in a connected $\hat{S}$-neighbourhood of $\hat{z}$. Take a suitably small such neighbourhood $\hat{U}$. For all $\hat{t} \in \hat{U}$,
\begin{equation*}
\gamma \cdot q(\hat{t}) = q(d(\hol^{\hat{S}}(\hat{\lambda}_0)(\hat{t}), \hat{t}) \cdot \hat{t}) = q(\hol^{\hat{S}}(\hat{\lambda}_0)(\hat{t})) = q(\hat{t}).
\end{equation*}
Therefore $\gamma$ acts as the identity on the open neighbourhood $q(\hat{U})$ of $q(\hat{z})$ in $M$, as required.

Now, the source of $q_1(\hat{y}, [\lambda], \hat{x})$ is $q(\hat{z}) = q(\hat{x})$, and its target is $d(\hat{y}, \hat{z}) \cdot q(\hat{z}) = q(\hat{y})$. It is tedious but straightforward to check that $q_1$ respects multiplication, and that $q_1$ is smooth. Thus $(q, q_1)$ is a smooth functor.

Finally, we show the functor $(q, q_1)$ is a Morita map. There are three things to check:
\begin{itemize}
\item The map $\overline{q}:\hat{X}/p^*(\Hol(\cl{F})) \to M/\Gamma$ is a homeomorphism. To see this, observe that $\overline{q}$ factors as $\overline{p}$ followed by the map $X/\cl{F} \to M/\Gamma_{\text{eff}}$ from Proposition \ref{prop:develop-leaf-space}, and both of these are diffeomorphisms, hence homeomorphisms.
\item The homomorphism $q_{\hat{x}}:\Hol(\cl{F})_{p(\hat{x})} \to G_{q(\hat{x})}$ is an isomorphism. We give its inverse. Let $\germ_{q(\hat{x})} h(\gamma) \in G_{q(\hat{x})}$. Then $\gamma \cdot \hat{x}$ and $\hat{x}$ are in the same leaf of $\hat{\cl{F}}$. Let $\hat{\lambda}$ be any leafwise path joining them, and set $\lambda := p_*\hat{\lambda}$. Then $\lambda$ is a loop at $p(\hat{x})$, and we take the inverse of $q_{\hat{x}}$ to be $\gamma \mapsto [\lambda]$. This is well-defined because if $\hat{\lambda}$ is any leafwise loop at $\hat{x}$, then $p_*\hat{\lambda}$ has trivial holonomy: the foliation $\hat{\cl{F}}$ is simple, so $\hat{\lambda}$ has trivial holonomy, hence so does $p_*\hat{\lambda}$.
\item The normal derivative $N_{\hat{x}}q:N_{\hat{x}} \hat{L}_{\hat{x}} \to T_{q(\hat{x})}M$ is an isomorphism. This follows immediately from the fact $\hat{\cl{F}}$ is simple.
\end{itemize}
\end{proof}

\begin{rem}
  In the set-up of Proposition \ref{prop:morita-to-pseudogroup}, choose a complete transversal $S$ to $\cl{F}$. The Morita equivalence between $\Hol(\cl{F})$ and the germ groupoid of the pseudogroup generated by the $\Gamma$ action on $M$ induces an equivalence of the pseudogroups $\Psi(\Hol_S(\cl{F}))$ and the pseudogroup generated by the $\Gamma$ action. Therefore, up to equivalence, the holonomy pseudogroup of $\cl{F}$ is generated by the $\Gamma$ action.
\end{rem}

In Proposition \ref{prop:morita-to-pseudogroup}, we encounter the pseudogroup generated by the $\Gamma$ action. In general, the corresponding germ groupoid will not be Morita equivalent to the action groupoid $\Gamma \ltimes M$. The obstruction is that two distinct elements of $\Gamma$ may induce diffeomorphisms with the same germ at some point. When $(X, \cl{F}, p)$ is a developable Riemannian foliation, however, this obstruction vanishes. We will first use the following lemma.

\begin{lem}
  \label{lem:germ-determined}
  If $\Gamma$ is a countable group acting effectively on a connected Riemannian manifold $M$ by isometries, then $\Gamma \ltimes M$ is Morita equivalent to $G$, the germ groupoid of the pseudogroup generated by the $\Gamma$ action.
\end{lem}
\begin{proof}
Let $h:\Gamma \to \Diff(M)$ be the action homomorphism.  We claim the functor $\phi: \Gamma \ltimes M \to G$ given by the identity on the base, and on arrows by
  \begin{equation*}
    (\gamma, m) \mapsto \germ_m h(\gamma),
  \end{equation*}
  is a Morita map.

The functor $\phi$ is smooth because $\Gamma$ being countable implies $\Gamma \ltimes M$ is \'{e}tale. Since $\phi$ is the identity on the base, it descends to the identity $M/\Gamma \to M/G$, which is a homeomorphism, and the normal derivative of $\phi$ is an isomorphism. It remains to show that the homomorphism $\phi_m:\Gamma_m \to G_m$ is an isomorphism. It is onto by definition of $G_m$. To see it is one-to-one, suppose that $\germ_m h(\gamma) = \germ_m h(\gamma')$. Because $h(\gamma)$ and $h(\gamma')$ are isometries, the fact their germs agree at $m$ implies that $h(\gamma) = h(\gamma')$ on the connected component of $m$, which is $M$ (see \cite[Proposition 5.22]{L18}). Since $\Gamma$ acts effectively, we conclude that $\gamma = \gamma'$.
\end{proof}

Now let $(X, \cl{F}, p, g)$ be a Riemannian developable foliation. Let $\Gamma_M$ be the normal subgroup of $\Gamma$ consisting of elements that act trivially on $M$, and set $\Gamma_{\text{eff}} := \Gamma/\Gamma_M$.  Then $\Gamma_{\text{eff}}$ naturally acts on $M$ effectively.
The transverse metric $g$ pulls back along $p$ to a transverse metric for $\hat{\cl{F}}$, and this descends to a Riemannian metric $g_M$ on $M$ for which $\Gamma_{\text{eff}}$ acts by isometries. We now get a corollary to Proposition \ref{prop:morita-to-pseudogroup}.

\begin{cor}
  \label{cor:effective-germ-determined-equivalence}
  If $(X, \cl{F}, p, g)$ is a developable Riemannian foliation, and the cover $p$ is Galois and connected, then $\Hol(\cl{F})$ is Morita equivalent to $\Gamma_{\text{eff}} \ltimes M$.
\end{cor}
\begin{proof}
  By Proposition \ref{prop:morita-to-pseudogroup}, we know that $\Hol(\cl{F})$ is Morita equivalent to $G$, the germ groupoid of the pseudogroup generated by the $\Gamma$ action on $M$. This pseudogroup is also generated by the $\Gamma_{\text{eff}}$ action. But since $\Gamma_{\text{eff}}$ is a countable group acting effectively on $M$ by isometries (for the metric $g_M$ described in the paragraph above), and $M$ is connected because $X$ is connected, we conclude by Lemma \ref{lem:germ-determined} that $G$ is Morita equivalent to $\Gamma_{\text{eff}}\ltimes M$.
\end{proof}
\subsection{Lie foliations}\label{Darboux-covering}
Our main example of developable foliations will be $\fk{g}$-Lie foliations.
\begin{defn} \label{Lie-foliation} Let $(X, \cl{F})$ be a co-dimension $q$ foliation of a connected manifold, and let $\mathfrak{g}$ be a $q$-dimensional real Lie algebra. We say that $(X, \mathcal{F})$ is a \define{complete} $\mathfrak{g}$-\define{Lie foliation} if it has a transverse parallelism $\{\bar{Z}_1, \cdots, \bar{Z}_q\}$, called a complete transverse $\mathfrak{g}$-Lie parallelism, such that:
\begin{enumeratea}
\item we have
  \begin{equation*}
    [\bar{Z}_i, \bar{Z}_j]=\displaystyle \sum_{p}c_{ij}^p \bar{Z}_p,
  \end{equation*}
where the $c_{ij}^p$ are the structure constants of $\mathfrak{g}$;

\item $\bar{Z}_1, \cdots, \bar{Z}_q$ are represented by complete foliate vector fields $Z_1, \cdots, Z_q$ on $X$.

\end{enumeratea}
As a convention, we identify $\mathfrak{g}$ with the Lie sub-algebra of $\mathfrak{X}(X, \mathcal{F})$ generated by $\{\bar{Z}_1, \cdots, \bar{Z}_q\}$.
 \end{defn}

Fedida described the structure of complete $\fk{g}$-Lie foliations in \cite{F71}. We review their construction, as retold in \cite{Mo88}, and summarize it in a theorem. Let $\{\bar{Z}_1, \cdots, \bar{Z}_q\}$ be a complete transverse $\mathfrak{g}$-Lie parallelism of $(X,\mathcal{F})$, and let $Z_1, \cdots, Z_q$ be complete foliate vector fields representing the transverse vector fields $\bar{Z}_1, \cdots, \bar{Z}_q$. Define a $\mathfrak{g}$-valued $1$-form $\alpha$ on $X$ as follows: any vector $Z_x\in T_xX$ can be uniquely written as
 \begin{equation*}
   Z_x= \lambda_1 (Z_1)_x+\cdots+\lambda_q (Z_q)_x+Z_x',
 \end{equation*}
 where $Z_x' \in T_x\mathcal{F}$, and define
 \begin{equation*}
   \alpha (Z_x) := \lambda_1 \bar{Z}_1+\cdots+\lambda_q \bar{Z}_q.
 \end{equation*}

 Let $L_{\alpha}$ be the Lie subalgebra of $\mathfrak{R}(\mathcal{F})$ formed by those foliate vector fields whose corresponding transverse vector fields belong to $\mathfrak{g}$ (viewed as a subalgebra of $\fk{X}(X, \cl{F})$), and let $G$ be the unique connected and simply-connected Lie group whose Lie algebra is $\mathfrak{g}$. For each $Z \in L_{\alpha}$, we can define its lift $\ti{Z}$ to $\fk{X}(X \times G)$ as follows:
\begin{equation}\label{lift}
  \ti{Z}_{(x, g)} := Z_x + (L_g)_*(\alpha(Z_x)) \in T_{(x, g)}(X\times G)\cong T_x X \oplus T_gG,
\end{equation}
where $(L_g)_*: \fk{g} \to T_gG$ is the tangent map of left multiplication by $g$. The lifted vector field $\ti{Z}$ is invariant under the left action of $G$ on $X \times G$. Moreover, the correspondence
\begin{equation*}
  \fk{R}(\cl{F}) \to \fk{X}(X \times G), \quad Z \mapsto \ti{Z} 
\end{equation*}
is a Lie algebra homomorphism that lifts $L_{\alpha}$ to a Lie algebra $\ti{L}_{\alpha}$. 

Let $\ti{H}$ be the distribution over $X \times G$ spanned by the Lie algebra $\ti{L}_\alpha$. This is an involutive distribution, of the same dimension as $\cl{F}$, and thus induces a foliation on $X \times G$. This foliation is invariant under the left $G$-action on $X \times G$. Fix a leaf $\ti{X}$ of this foliation. Then we have the following diagram:
\begin{equation}\label{Darboux-dia}
\begin{tikzcd}
  &\ti{X} \arrow[dl, swap, "\pr_1"] \arrow[dr, "\pr_2"]&\\ X & &G.
\end{tikzcd}
\end{equation}
The map $\pr_1: \ti{X}\rightarrow X$ is a covering map, called the \textbf{Darboux covering} of $X$, and $\pr_2: \ti{X}\rightarrow G$ is a fibration. The pullback foliation $\ti{\cl{F}} := \pr_1^{-1}(\cl{F})$ coincides with the foliation on $\ti{X}$ given by the fibers of $\pr_2$. It follows that $(X, \cl{F}, \pr_1)$ is a developable foliation, with developing map $\pr_2$.

Consider $(x_0, g_0) \in \ti{X}$. Define the subgroup $\Gamma$ of $G$ by
\begin{equation*}
   \Gamma := \{ g\in G\, \vert\, g \cdot (x_0, g_0) \in \ti{X}\}.
\end{equation*}
Since $\ti{X}$ is the leaf of a foliation that is invariant under the left action of $G$, its translate $g\cdot \ti{X}$ must also be a leaf of this foliation for each $g\in G$. Therefore, if $g \cdot (x_0, g_0) \in \ti{X}$, then $g \cdot \ti{X}\cap \ti{X}\neq \emptyset$, and hence $g\cdot \ti{X}=\ti{X}$ for each $g\in \Gamma$. This shows that $\Gamma$ does not depend on the choice of $x_0$. Furthermore we can identify $\Gamma$ with the Deck transformation group of the cover $\pr_1:\ti{X} \to X$, and $\Gamma$ acts transitively on each fiber of $\pr_2$. Thus $\pr_1:\ti{X} \rightarrow X$ is a Galois covering. We can summarize the preceding discussion in a theorem.

\begin{thm}[Fedida]\label{leaf-space}
  Let $(X, \cl{F})$ be a complete $\fk{g}$-Lie foliation. Then $\cl{F}$ is developable with respect to the Darboux cover $\pr_1:\ti{X} \to X$, and the developing map is $\pr_2:\ti{X} \to G$, where $G$ is the connected and simply-connected Lie group with Lie algebra $\fk{g}$. The cover $\pr_1$ is Galois, and its group of deck transformations $\Gamma$ is naturally a subgroup of $G$, and acts on $G$ effectively.
\end{thm}


To conclude this subsection, we explain how the Darboux cover lifts to the universal cover. Suppose that $p:\hat{X} \to X$ is the universal cover of $X$. Establish notation according to the following commutative diagram, which exists because $p$ is universal:
\begin{equation*}
  \begin{tikzcd}
    & \hat{X} \ar[d, "\rho"] \ar[ddl, bend right, "p"'] \ar[ddr, bend left, "q"] & \\
    & \ti{X} \ar[dl, "\pr_1"'] \ar[dr, "\pr_2"] & \\
    X & & G.
  \end{tikzcd}
\end{equation*}
The map $q: \hat{X}\rightarrow G$ is a Serre fibration. It follows from the long exact sequence of homotopy groups, and the fact $\hat{X}$ and $G$ are both connected and simply-connected, that each fibre of $q$ is connected and simply-connected. As a result, $q:\hat{X} \to G$ is a locally trivial fibration, whose fibers are the leaves of the lifted $\hat{\cl{F}} := p^{-1}(\cl{F})$. Thus $G$ is also the leaf space of the foliated manifold $(\hat{X},\hat{\mathcal{F}})$, and $(X, \cl{F}, p)$ is a developable foliation.
     
The map $p: \hat{X}\rightarrow X$ is the universal cover of $X$. Thus the covering map $\rho:\hat{X}\rightarrow \tilde{X}$ induces a surjective homomorphism $\psi: \pi_1(X)\rightarrow \Gamma \leq G$. As a result, $\pi_1(X)$ also acts on $G$ by $\tau \cdot g := \psi(\tau) g$. The map $q$ is $\pi_1(X)$-equivariant with respect to this action, $q(\tau \cdot \hat{x}) = \tau \cdot q(\hat{x})$. We summarize this in a proposition.
\begin{prop}
  \label{prop:lifted-fedida}
  Let $(X, \cl{F})$ be a complete $\fk{g}$-Lie foliation. Then it is developable with respect to the universal cover $p:\hat{X} \to X$. The developing map is $q:\hat{X} \to G$, and the action of $\pi_1(X)$ on $G$ is given by the homomorphism $\psi:\pi_1(X) \to \Gamma \leq G$ induced by the cover $\hat{X} \to \ti{X}$ of the Darboux cover $\ti{X}$. The leaves of $\hat{\cl{F}}$ are simply-connected.
\end{prop}
In particular, any foliated principal bundle $P \to \hat{X}$ over $(\hat{X}, \hat{\cl{F}})$ is transversal, by Example \ref{ex:simply-connected-leaves-transversal}, and the leaf space $P/\cl{F}_P$ is a principal bundle over $G$ by Lemma \ref{quotient-bundle}.

\section{The leaf space of a Killing foliation}
\label{sec:leaf-space-killing}
In this section we show that the leaf space of a complete Killing foliation is a diffeological quasifold. Note that by Proposition \ref{Molino-sheaf}, there is a natural transverse action of the structure Lie algebra on a leaf closure of a Killing foliation. We first make an observation about such an action.
  \begin{lem}\label{killing-action-isotropy} Let $(M, \mathcal{F})$ be a complete Killing foliation, let $\mathfrak{g}$ be its structure Lie algebra, and let $L$ be a leaf of $\mathcal{F}$. Consider the natural transverse Lie algebra action $a: \mathfrak{g} \rightarrow \mathfrak{X}(\overline{L}, \mathcal{F}_{\overline{L}})$.  Suppose that $\xi \in \text{stab}( \mathfrak{g}\ltimes \mathcal{F}, x_0)$ for a point $x_0$ in $ \overline{L}$. Then $\xi\in \text{stab}( \mathfrak{g}\ltimes \mathcal{F}, x)$ for every point $x\in \overline{L}$.  
  \end{lem}  
  \begin{proof}
    Keep the notation from Section \ref{Molino-theory}. Choose a point $\tilde{x}_0$ in the transverse orthonormal bundle $P$ such that  $\pi(\tilde{x_0})=x_0$. Let $L_P$ be a leaf of $\mathcal{F}_P$ that passes through $\tilde{x}_0$, and let $\overline{L_P}$ be the leaf closure of $L_P$. Note that by Lemma \ref{leaf-closure} the map $\pi: \overline{L_P}\rightarrow \overline{L}$ projects $\overline{L_P}$ onto $\overline{L}$. 

For  $\xi\in \mathfrak{g}$, the image of $a(\xi)$ under the lifting homomorphism \ref{lifting-homo} is a global transverse Killing vector field that is tangent to $\overline{L_P}$.  Now suppose that $\xi \in \text{stab}( \mathfrak{g}\ltimes \mathcal{F}, x_0)$.  Write $\bar{\xi}_P := \pi^{\sharp}a(\xi)$. Then there exists $\lambda \in \mathfrak{t}=\text{Lie}(K) $ such that $(\bar{\xi}_P)_{\tilde{x}_0}= (\bar{\lambda}_P)_{\tilde{x_0}} $, where $\bar{\lambda}_P$ is the transverse vector field represented by the fundamental vector field on $P$ corresponding to $\lambda \in \mathfrak{t}$.  To finish the proof, it suffices to show that $(\bar{\xi}_P)_{\tilde{x}}= (\bar{ \lambda}_P)_{\tilde{x}}$ for each $\tilde{x} \in \overline{L_P}$.

Choose a basis $\{\xi_1, \cdots, \xi_q\}$ of the structure Lie algebra $\mathfrak{g}$ such that $\{\pi^{\sharp}a(\xi_1), \cdots, \pi^{\sharp}a(\xi_q)\}$ are global sections of the Molino sheaf, and choose a connected simple open set $U \subseteq \ol{L_P}$ that contains $\tilde{x}_0$. Let $p: U\rightarrow \overline{U}$ be the local quotient map onto the leaf space $\ol{U} = U/\cl{F}_P|_U$, and let $\bar{x}_0=p(\tilde{x}_0)$. Then 
$\{ \pi^{\sharp}a(\xi_1), \cdots, \pi^{\sharp}a(\xi_q)\}$ project to $\overline{U}$ as $q$ many linearly independent vector fields $\{\xi_{1, \overline{U}}, \cdots,\xi_{q, \overline{U}}\}$ which satisfy $[\xi_{i, \overline{U}}, \xi_{j, \overline{U}}]=0$ for $1\leq i, j\leq q$, while $\bar{\xi}_{P}-\bar{\lambda}_P$ projects to $\overline{U}$ as a vector field $\zeta$ that satisfies $\zeta_{\bar{x}_0}=0$. By definition of the Molino sheaf, $[\pi^{\sharp}a(\xi_i), \bar{\xi}_P-\bar{\lambda}_P]=0$ for each $1\leq i\leq q$. It follows that $[\xi_{i, \overline{U}}, \zeta]=0$ for all $1\leq i\leq q$.  Since $\overline{U}$ is a $q$-dimensional connected manifold, we conclude that $\zeta$ vanishes everywhere on $\overline{U}$, and so $\bar{\xi}_P=\bar{\lambda}_P$ on $U$. This finishes the proof, since $U$ was arbitrarily chosen.
\end{proof}

Now we prove our main result.
\begin{thm}\label{thm:1}
  Suppose that $(M, \cl{F})$ is a complete Killing foliation of a connected manifold with structure algebra $\fk{g}$, such that the transverse action of $\fk{g}$ is also complete. Then $M/\cl{F}$ is a diffeological quasifold.

  More precisely, for each leaf $L$, with closure $X$, take a universal cover $p: \hat{X} \to X$. Let $\hat{E}$ be the pullback of $NX \to X$ along $p$. Then $(NX, \cl{F}_{NX}, \pr_2:\hat{E} \to NX)$ is a developable foliation, say with developing map denoted $q:\hat{E} \to E$. Furthermore, $(X, \cl{F}|_X)$ is a complete $\fk{h}$-Lie foliation, where
  \begin{equation*}
    \fk{h} := \fk{g}/\fk{k}, \text{ and } \fk{k}:=\{\xi \in \fk{g} \mid \xi \in \stab(\fk{g} \rtimes \cl{F}, x) \text{ for some } x \in X\}.
  \end{equation*}
Let $H$ be the connected and simply-connected Lie group with Lie algebra $\fk{h}$. Then $H = (\R^r, +)$, where $r$ is the codimension of $\cl{F}|_X$ in $X$. The space $E$ is a vector bundle over $H$ of rank equal to the codimension of $X$ in $M$. The action groupoid $\pi_1(X) \ltimes E$ is isomorphic to an action groupoid $\pi_1(X) \ltimes (H \times E_0)$, where $E_0$ is a fiber of $E$, and $\pi_1(X)$ acts affinely.
\end{thm}

Before we prove this theorem, we give an immediate corollary.

\begin{cor}
  \label{cor:morita-equivalence}
   In the setting of Theorem \ref{thm:1}, there exists a tubular neighbourhood $V$ of $L$ such that the restricted holonomy groupoid $\Hol(\cl{F})|_V$ is Morita equivalent to $\pi_1(X)_{\text{eff}} \ltimes (H \times E_0)$, where $\pi_1(X)_{\text{eff}}$ is the quotient of $\pi_1(X)$ by the normal subgroup of elements which act trivially on $E$.
 \end{cor}
 \begin{proof}
By Proposition \ref{tubular}, choose a tubular neighbourhood $V$ so that the foliation $(V, \cl{F}|_V)$ is isomorphic to $(NX, \cl{F}_{NX})$. Then $\Hol(\cl{F}|_V) \cong \Hol(\cl{F})|_V \cong \Hol(\cl{F}_{NX})$. Then, since $(NX, \cl{F}_{NX})$ is Riemannian, we simply apply Corollary \ref{cor:effective-germ-determined-equivalence} to the developable foliation $(NX, \cl{F}_{NX}, \pr_2)$ from Theorem \ref{thm:1}.
 \end{proof}
Now we prove the theorem.  
 \begin{proof}[proof of Theorem \ref{thm:1}]
   First, assuming the more precise clause of Theorem \ref{thm:1}, we will show that $M/\cl{F}$ is a diffeological quasifold. Let $L \in \cl{F}$, with closure $X$, and take a universal cover $p:\hat{X} \to X$. Apply Proposition \ref{prop:develop-leaf-space} to $(NX, \cl{F}_{NX}, \pr_2)$ to see that $NX/\cl{F}_{NX}$ is diffeomorphic to $E/\pi_1(X)$. But then $E/\pi_1(X)$ is diffeomorphic to $(H \times E_0)/\pi_1(X)$, which is a model quasifold. We conclude that $NX/\cl{F}_{NX}$ is a diffeomorphic to a model diffeological quasifold. By Proposition \ref{tubular}, the foliation $(NX, \cl{F}_{NX})$ is isomorphic to $(V, \cl{F}|_V)$ for a tubular neighbourhood $V$ of $L$. Therefore $NX/\cl{F}|_{NX} \cong V/\cl{F}|_V$, and $M/\cl{F}$ is a diffeological quasifold.
   
   Now we prove the precise statement in Theorem \ref{thm:1}. Fix a leaf $L$, with closure $X$. Let
   \begin{equation}\label{eq:stab-lie-alg}
     \fk{k} := \{\xi \in \fk{g} \mid \xi \in \stab(\fk{g} \rtimes \cl{F}, x) \text{ for some } x \in X\}.
   \end{equation}
By Lemma \ref{killing-action-isotropy}, if $\xi \in \stab(\fk{g} \rtimes \cl{F}, x)$ for some $x \in X$, then it is in $\stab(\fk{g} \rtimes \cl{F}, x')$ for every $x' \in X$. Therefore, $\fk{k}$ is a well-defined Lie subalgebra of $\fk{g}$. Its quotient $\fk{h} := \fk{g}/\fk{k}$ is a Lie algebra, and the foliation $(X, \cl{F}|_X)$ is a complete $\fk{h}$-Lie foliation. Specifically, if we identify $\fk{h}$ with a complement to $\fk{k}$ in $\fk{g}$, so that $\fk{g} = \fk{k} \oplus \fk{h}$, and we fix a basis $\{\xi_1,\ldots, \xi_r\}$ of $\fk{h}$, then the transverse vector fields corresponding to the $\xi_i$ (under the transverse action of $\fk{g}$ on $X$) give a complete transverse $\fk{h}$-Lie parallelism for $\cl{F}|_X$. In particular, $r$ is the codimension of $\cl{F}|_X$ in $X$.

Being a complete $\fk{h}$-Lie foliation, we apply Proposition \ref{prop:lifted-fedida} to $(X, \cl{F}|_X)$. First, to keep notation simple, rename $\cl{F}|_X$ to $\cl{F}$. Fix $H$ to be the connected and simply-connected Lie group whose Lie algebra is $\fk{h}$, and choose a universal cover $p:\hat{X} \to X$. By Proposition \ref{prop:lifted-fedida}: $(X, \cl{F}, p)$ is developable, and moreover the leaves of $\hat{\cl{F}} := p^{-1}(\cl{F})$ are simply-connected.


   Now let $NX$ be the normal bundle of $X \subseteq M$, equipped with the lifted foliation $\cl{F}_{NX}$.  Let $\hat{E}$ be the pullback of $NX$ along $p:\hat{X} \to X$, and let $\cl{F}_{\hat{E}}$ be the pullback of $\cl{F}_{NX}$ by the second projection $\pr_2:\hat{E} \to NX$. The foliation $(\hat{E}, \cl{F}_{\hat{E}})$ makes the vector bundle $\hat{E} \to \hat{X}$ foliated. Moreover, by Example \ref{ex:simply-connected-leaves-transversal}, this vector bundle is transversal, because $(\hat{X}, \hat{\cl{F}})$ has simply-connected leaves. Since $(\hat{X}, \hat{\cl{F}})$ is also simple and $\hat{E} \to \hat{X}$ is transversal, by Lemma \ref{quotient-bundle} the leaf space $E := \hat{E}/\cl{F}_{\hat{E}}$ is naturally a vector bundle over $H$. Keeping track of dimensions, we see $E$ has rank equal to the codimension of $X$ in $M$. We have the diagram
   \begin{equation}\label{eq:7}
     \begin{tikzcd}
       & \hat{E} \ar[dl, "\pr_2"'] \ar[dr, "q \text{ := quotient}"] & \\
       NX & & E,
     \end{tikzcd}
   \end{equation}
   and we see that the foliation $(NX, \cl{F}_{NX}, \pr_2)$ is developable.
   
   The group $\pi_1(X)$ of deck transformations of $p$ coincides with the group of deck transformations of $\pr_2:\hat{E} \to NX$, by setting
   \begin{equation*}
     \tau \cdot (\hat{x}, e) := (\tau \cdot \hat{x}, e).
   \end{equation*}
   This action descends to the leaf space $E$, and $q:\hat{E} \to E$ is $\pi_1(X)$-equivariant. We will now describe how to identify $E$ with Cartesian space. First, consider the transverse action $a:\fk{h} \to \fk{X}(NX, \cl{F}_{NX})$. For $\xi \in \fk{h}$, represent $a(\xi)$ by a complete foliate vector field $\xi_{NX}$, with flow $\varphi_t$. This is a one-parameter group of foliated diffeomorphisms of $(NX, \cl{F}_{NX})$, which corresponds to a one-parameter group of vector bundle automorphisms (not necessarily lifting the identity) of $E$; we take the action of $\xi$ on $E$ to be given by the vector field $\xi_E$ corresponding to this one-parameter group. Thus we have an action of $\fk{h}$ on $E$.  By the Lie-Palais theorem, the $\fk{h}$ action exponentiates to an action of $H$ on $E$. Denote the action map by
   \begin{equation*}
   F:H \to \Aut(E). 
   \end{equation*}

   Because $\fk{h} = \fk{g}/\fk{k}$ is abelian, $H$ is the additive group $(\R^r, +)$. Let $E_0$ be the fiber of $E$ over $0 \in H$, which is a vector space. We have a global trivialization of $E$ given by
   \begin{equation*}
     \Phi:E \to H \times E_0, \quad \Phi(e) := (h, F(h^{-1})(e)), \text{ where } e \in E_h.
   \end{equation*}
   The group $\pi_1(X)$ acts on $E_0$ linearly by
   \begin{equation*}
     \tau \cdot v := F(\psi(\tau)^{-1})(\tau \cdot v),
   \end{equation*}
where $\psi:\pi_1(X) \to H$ is defined in the end of Subsection \ref{Darboux-covering} and the right $\tau \cdot v$ denotes the action of $\pi_1(X)$ on $E$. Then, the trivialization $\Phi$ is $\pi_1(X)$-equivariant, with respect to the action of $\pi_1(X)$ on $H \times E_0$ given by
   \begin{equation*}
     \tau \cdot (h, v) := (\psi(\tau) + h, \tau \cdot v)
   \end{equation*}
   This action is an affine action of the countable group $\pi_1(X)$ on $H \times E_0$. Noting that $\Phi$ induces an isomorphism $\pi_1(X) \ltimes E \cong \pi_1(X) \ltimes (H \times E_0)$ completes the proof.
 \end{proof}

 If we do not require that $(M, \cl{F})$ is Killing or Riemannian, then Theorem \ref{thm:1} will not hold. The foliation described in \cite[Section 7]{KM22} is an explicit example. It is the suspension of a non-identity diffeomorphism $h: \R \to \R$ whose jet at $0$ coincides with the jet of the identity at $0$. See \cite[Proposition 7.2]{KM22} for the details, but the essential argument is that if the leaf space $M/\cl{F} \cong \R/h$ were a diffeological quasifold, then any lift of the identity $\R/h \to \R/h$ whose jet at $0$ coincides with the identity, must be the identity. This is because affine transformations are determined by their 1-jet at a point. In the upcoming Subsection \ref{sec:comp-haefl-model}, we show that Theorem \ref{thm:1} may fail even if we require $(M, \cl{F})$ to be Riemannian. 

 We now apply our theorem to some examples. Throughout, $(M, \cl{F})$ denotes a complete Killing foliation with structure algebra $\fk{g}$, such that the transverse action of $\fk{g}$ is complete.
 
 \begin{ex}\label{ex:dense-leaf}
   Suppose that $(M, \cl{F})$ has a dense leaf $L$. Then its closure $X$ is all of $M$, and the normal bundle $NX$ is trivial. Let $p:\hat{M} \to M$ be a universal cover of $M$. Then the bundles $\hat{E}$ and $E$ are also trivial, so Theorem \ref{thm:1} gives that $(M, \cl{F}, p)$ is developable, with developing map $q:\hat{M} \to H$. We have $M/\cl{F} \cong H/\pi_1(X)$, and the action of $\pi_1(X)$ on $H$ is isomorphic to an affine action on $H$. Furthermore, since $H$ is the leaf space of the Darboux cover $(\ti{M}, \ti{\cl{F}})$, and $\pi_1(X)$ acts on $H$ through $\Gamma \leq H$, to compute $H/\pi_1(X)$ it suffices to find the Darboux cover of $(M, \cl{F})$.

   As a special case, consider the torus $T^2 = \R^2/\Z^2$, and $\cl{F}_\lambda$ the foliation on $T^2$ induced by the foliation of $\R^2$ by lines of slope $\lambda$, for $\lambda$ irrational. This is a complete $\fk{h}$-Lie foliation, with $\fk{h} := (\R, +)$. Working through Subsection \ref{Darboux-covering}, using the transverse parallelism induced by the foliate vector field $Z_1 := a\frac{\partial}{\partial x} + b \frac{\partial}{\partial y}$ where $\frac{b}{a} = -\frac{1}{\lambda}$ and $a^2+b^2 = 1$, we find that the Darboux cover is
   \begin{equation*}
     \pr_1: \hat{X} := \{([at - bs, bt + as], t)\} \to T^2.
   \end{equation*}
   The group $\pi_1(T^2)$ acts on $\R$ through the subgroup $\Gamma \leq \R$ consisting of those $g$ fixing $\hat{X}$. Using the condition that
   \begin{equation*}
     g \cdot \hat{X} = \hat{X} \iff g \cdot ([0,0], 0) = ([0,0],g) \in \hat{X},
   \end{equation*}
we find that $\Gamma = a \Z + b \Z$. Therefore $T^2/\cl{F}_\lambda \cong \R/\Gamma$, and one can also show that $\R/\Gamma \cong \R/(Z+ \lambda \Z)$. While we worked in dimension 2 for convenience, this argument extends to higher dimensions when we replace $T^2$ with $T^n$, and the foliation $\cl{F}_\lambda$ by $\cl{F}_H$, where $H$ is a hyperplane in $\R^n$ whose parallel translates determine the foliation $(T^n, \cl{F}_H)$. In this case, we find that $T^n/\cl{F}_H$ is diffeomorphic to $\R/\Gamma$, for appropriately chosen $\Gamma$. This reproduces the results in \cite{IL90}, from a different perspective.
 \end{ex}

 \begin{ex}
   Suppose that a leaf $L$ of $(M, \cl{F})$ is compact. Its closure $X$ is again $L$. Therefore the foliation $\cl{F}|_X$ has one leaf, namely $L$, and its lift $\hat{\cl{F}}$ to $\hat{X}$ also has one leaf, namely $\hat{X}$. It follows that $H = \hat{X}/\hat{\cl{F}}$ is the trivial group. The vector bundle $NX$ is the restriction $N\cl{F}|_L$. The vector bundle $E \to H$ is a vector space $E \to \{*\}$, which we can identify with $N_{x_0}L$ by
   \begin{equation*}
     N_{x_0}L \to \hat{E}/\cl{F}_{\hat{E}}, \quad v \mapsto L_{(\hat{x}_0, v)},
   \end{equation*}
   where $\hat{x}_0$ is the constant loop at $x_0$. Under this identification, $\pi_1(X, x_0)$ acts on $N_{x_0}L$ by linearized holonomy:
   \begin{equation*}
     \tau \cdot v := d_{x_0}\Hol(\tau)(v).
   \end{equation*}
   Therefore, a neighbourhood of $L$ in the leaf space $M/\cl{F}$ is diffeomorphic to $N_{x_0}L/\pi_1(X, x_0) = N_{x_0}L/\Hol(L, x_0)$.
   By \cite[Theorem 2.6]{MM03}, if every leaf of a Riemannian foliation is compact, then every leaf has finite holonomy. Therefore in this case each $N_{x_0}L/\Hol(L, x_0)$ is a model orbifold, and $M/\cl{F}$ is an orbifold.
 \end{ex}
 
 \subsection{Comparison to Haefliger's model}
\label{sec:comp-haefl-model}

In \cite{H88}, Haefliger gave a model describing an arbitrary complete pseudogroup of isometries. In this subsection, we describe Haefliger's model in the Killing case, and compare it to our Theorem \ref{thm:1}. Fix a Killing foliation $(M, \cl{F})$, with structure group $\fk{g}$. Take a complete transversal $S$ to $\cl{F}$, and $x \in S$. Haefliger builds his model from the following ingredients:

\begin{itemize}
 \item The Lie algebra $\fk{g}$, which we understand as $\mbf{C}(M, \cl{F})_x \cong \mbf{C}(S, \Psi_S(M, \cl{F}))_x$ (see Subsection \ref{sec:molin-sheaf-pseud}).
 \item The action of $K := (\ol{\Psi_S(M, \cl{F})})_x$ on $\fk{g}$, as given in \eqref{eq:action-on-stalk}.
 \item The inclusion $i:\fk{k} \hookrightarrow \fk{g}$, where $\fk{k}$ is the Lie algebra of $K$.
 \item The normal space $V$ at $x$ to the orbit $\ol{\Psi_S(M, \cl{F})} \cdot x$. Note this orbit is also the intersection $\ol{L_x} \cap S$.
 \item The action of $K$ on $V$ by differentiation: $k \cdot v := k_* v$.
 \end{itemize}

 Haefliger constructs a pseudogroup equivalent to the restriction of $\Psi_S(M, \cl{F})$ to a tubular neighbourhood of the orbit $\Psi_S(M, \cl{F})\cdot x$ in the following way. Because $\fk{g}$ is abelian, we have $\fk{g} \cong \R^d$ for some $d$. Let $\ti{G} := (\R^d, +)$ be the unique simply-connected Lie group with Lie algebra $\fk{g}$. Then $\fk{k}$, being a subalgebra of $\fk{g}$, is isomorphic to $\R^m$ for some $m$. Let $K^\circ$ denote the identity component of $K$, and $\ti{K} := (\R^m, +)$ be the unique simply-connected Lie group with Lie algebra $\fk{k}$. We know $\ti{K}$ covers $K^\circ$; write $p:\ti{K} \to K^\circ$ for the covering map. Since $\fk{k} \cong \ti{K}$ and $\fk{g} \cong \ti{G}$, the inclusion $\iota: \fk{k} \hookrightarrow \fk{g}$ induces a homomorphism $j:\ti{K} \hookrightarrow \ti{G}$.

Since $K$ acts on $\fk{g}$, it also acts on $\ti{G}$, giving the semidirect product $\ti{G} \rtimes K$. Consider the quotient of $\ti{G} \rtimes K$ by the subgroup
 \begin{equation*}
   L := \{(j(k^{-1}), p(k)) \mid k \in \ti{K} \subseteq \ti{G}\}.
 \end{equation*}
 Under the conditions
 \begin{enumerate}[label = (\alph*)]
 \item $j(\ker p)$ is a discrete subgroup of $G_0$, and
   \item $\ker j \subseteq \ker p$,
 \end{enumerate}
which in this case hold because $j$ is an inclusion of $\R^m$ into $\R^n$, by \cite[Theorem 5.2]{H88}: the subgroup $L$ is closed, so we can form the quotient Lie group $G := (\ti{G} \rtimes K) / L$, and the map
 \begin{equation*}
   K \to G, \quad k \mapsto [e, k]
 \end{equation*}
 is a well-defined homomorphism and embedding of $K$ into $G$, so we view $K$ as a (compact) Lie subgroup of $G$. We therefore have the principal $K$ bundle $G \to G/K$, and combined with the representation of $K$ on $V$, we form the associated bundle $G \times_K V \to G/K$. Haefliger proves:
 \begin{itemize}
 \item the pseudogroup $\ol{\Psi_S(M, \cl{F})}$, restricted to a tubular neighbourhood, is equivalent to the pseudogroup generated by the action of $G$ on $G \times_K V$;
   \item the pseudogroup $\Psi_S(M, \cl{F})$, restricted to a tubular neighbourhood, is equivalent to the pseudogroup generated by the action of a dense subgroup $\Lambda$ of $G$ on $G \times_K V$. The group $\Lambda$ can be viewed as the fundamental group of $\Psi_S(M, \cl{F})$.
 \end{itemize}

 Now we compare our model with Haefliger's. First, Haefliger's Lie algebra $\fk{k}$ is exactly the Lie algebra $\fk{k}$ defined in \eqref{eq:stab-lie-alg}. Then Haefliger's $\fk{g}/\fk{k}$ coincides with our $\fk{h} = \fk{g}/\fk{k}$, and therefore we can identify $G/K$ and $H$. Indeed, $G/K$ is simply connected, being the quotient of $(\R^d, +)$ by a Lie subgroup, and its Lie algebra is $\fk{g}/\fk{k} \cong \fk{h}$. Furthermore, the vector bundle $G \times_K V$ has rank $\dim V$, which is also the rank of the bundle $E$ appearing in Theorem \ref{thm:1}. Since $G/K \cong H$ is contractible, $G\times_KV$ and $E$ are isomorphic vector bundles. However, whereas we have a global trivialization of $E$ that is $\pi_1(\ol{L})$-equivariant, we do not immediately find a $\Lambda$-equivariant global trivialization of $G \times_K V$. Indeed, the natural guess
 \begin{equation*}
   G\times_K V \to G/K \times V, \quad [g,v] \mapsto ([g],v)
 \end{equation*}
 is not well-defined if the action of $H$ on $V$ is non-trivial. Moreover, the action of $\Lambda$ on $G\times_K V$ is effective, but the action of $\pi_1(\ol{L})$ on $E$ is not generally effective.

 \begin{rem}
   If we drop the assumption that $\cl{F}$ is Killing, we can still carry through Haefliger's construction up to verifying conditions (a) and (b). In general, these conditions will not hold, and it takes Haefliger more work to get the local model of $\Psi_S(M, \cl{F})$. We saw above that a consequence of (a) and (b) is that the pseudogroup $\Psi_S(M, \cl{F})$ is locally equivalent to a pseudogroup generated by the action of a group, in Haefliger's case $\Lambda$ and in our case $\pi_1(\ol{L})$ (by Corollary \ref{cor:morita-equivalence}).\footnote{A Morita equivalence of \'{e}tale Lie groupoids induces an equivalence of their associated pseudogroups.} Haefliger shows in \cite[Corollary 5.6.2]{H88} that (a) and (b) are in fact equivalent to $\Psi_S(M, \cl{F})$ being locally equivalent to a pseudogroup generated by a group action. This means that any complete Riemannian foliation whose holonomy pseudogroup fails (a) or (b) cannot satisfy the conclusion of Corollary \ref{cor:morita-equivalence} for any group.
 \end{rem}
 
\section{Diffeological quasifolds as leaf spaces of foliations}
\label{sec:diff-quas}

In this section, we give a partial converse to Theorem \ref{thm:1}: we show that every diffeological quasifold is the leaf space of a foliation (Corollary \ref{cor:quas-is-leaf-space}), and give sufficient conditions for the associated foliation to be Riemannian (Corollary \ref{cor:riemannian-quas-is-leaf-space}). First, given a diffeological quasifold $X$, we may take its structural pseudogroup $\Psi$, as described in Subsection \ref{sec:struct-pseud-quas}. This is a countably-generated pseudogroup, so we begin with the following proposition.

\begin{prop}
  \label{prop:quastofol}
  Fix a manifold $N$, let $\{\psi_i\}_{i=1}^\infty$ be a countable set of transitions of $N$, and let $\Psi$ denote the pseudogroup generated by the $\psi_i$. There is a foliated manifold $(M, \mathcal{F})$ whose holonomy pseudogroup is (isomorphic to) $\Psi$. In particular, $M/\mathcal{F} \cong N/\Psi$. If $N$ has a Riemannian metric for which the $\psi_i$ are local isometries, then $(M, \mathcal{F})$ is a Riemannian foliation. 
\end{prop}

Except for the last assertion, this is due to Moerdijk and Mrc\u{u}n \cite[Example 5.23 (2)]{MM03}, who learned of the construction from Pradines, who in turn attributed it to  Hector. We give construction here in more detail, and justify the last claim.
\begin{proof}
  Set
  \begin{equation*}
    V := (N \times \R \times (0,1)) \cup \bigcup_{i=1}^\infty (\dom \psi_i \times (i,i+1) \times (0,3)).
  \end{equation*}
  This is an open subset of $N \times \R \times \R$. Define the equivalence relation ${\sim}$ on $V$ by
  \begin{equation*}
    (x,s,t) \sim (\psi_i(x), s, t-2) \text{ if } s \in (i, i+1) \text{ and } x \in \dom \psi_i \text{ and } t\in (2,3).
  \end{equation*}
  Each element of $V$ is equivalent to at most one other element. The quotient space $M := V/{\sim}$ is a manifold, and the quotient map is a smooth submersion; we check the details in Lemma \ref{lem:godemont} below. Let $\mathcal{F}_0$ be the foliation of $V$ given by the fibers of the first projection $\operatorname{pr}_1:V \to N$. The quotient $\pi:V \to M$ sends leaves of $\mathcal{F}_0$ to leaves of a foliation $\mathcal{F}$ on $M$. We claim $(M, \mathcal{F})$ is the desired foliation.

    We describe a Haefliger cocycle for $\mathcal{F}$. Let
   \begin{align*}
    &V^\Top_i := \dom \psi_i \times (i,i+1) \times (1.5,3), \quad V^\Mid_i := \dom \psi_i \times (i,i+1) \times (1,2), \\ &V^\Bot_i:= \dom \psi_i \times (i,i+1) \times (0,1.5), \quad   V^\Bot := N \times \R \times (0,1).
   \end{align*}
   Here ``top'', ``bot'', and ``mid'' stand for ``top,'' ``bottom,'' and ``middle,'' respectively. Define the map $\sigma^\Bot$ by the following diagram
\begin{equation*}
  \begin{tikzcd}
    V^\Bot \ar[d, "\pi"] \ar[dr, "{\operatorname{pr}_1}"] & \\
    \pi(V^\Bot) \ar[r, "\sigma^\Bot"']& N.
  \end{tikzcd}
\end{equation*}
The map $\sigma^\Bot$ is a well-defined submersion because the restriction of $\pi$ in the diagram is a diffeomorphism, and $\operatorname{pr}_1$ is a submersion. Writing $[x,s,t]$ for $\pi(x,s,t)$, we have an explicit expression for $\sigma^\Bot$:
\begin{equation*}
  \sigma^\Bot([x,s,t]) =
  \begin{cases}
    x &\text{if } t \in (0,1) \\
    \psi_i(x) &\text{if } s \in (i,i+1) \text{ and } t \in (2,3).
  \end{cases}
\end{equation*}
We similarly define the submersions $\sigma_i^\bullet:\pi(V^\bullet_i) \to N$, for $\bullet \in \{\Top,\Mid,\Bot\}$.  For $s \in (i, i+1)$, we have

\begin{align*}
    \sigma^\Bot_i([x,s,t]) &=
                       \begin{cases}
                         x &\text{if } t \in (0,1.5) \\
                         \psi_i(x) &\text{if } t \in (2,3),
                       \end{cases} \\
    \sigma^\Mid_i([x,s,t]) &= x, \\
  \sigma^\Top_i([x,s,t]) &=
                       \begin{cases}
                         \psi_i^{-1}(x) &\text{if } t \in (0,1) \\
                         x &\text{if } t \in (1.5,3).
                       \end{cases}
\end{align*}
The map $\sigma_i^\Top$ is well-defined because if $[x,s,t] \in \pi(V^\Top_i)$, then a representative $(x,s,t)$ is in either $\dom \psi_i \times (i,i+1) \times (1.5, 3)$, or in $\operatorname{im} \psi_i \times (i,i+1) \times (0,1)$.

Let $M_i^\bullet := \pi(V^\bullet_i)$, and $M^\Bot:= \pi(V^\Bot)$. A Haefliger cocycle for $\mathcal{F}_0$ is $\{(V, \operatorname{pr}_1, \operatorname{id})\}$. Because the $\sigma_i^\bullet$ and $\sigma^\Bot$ factor $\operatorname{pr}_1$ along $\pi$, and $\{M_i^\bullet\} \cup \{M^\Bot\}$ covers $M$, the $\sigma_i^\bullet$ induce a Haefliger cocycle representing $\mathcal{F}$. The only non-trivial elements of this cocycle are indicated in the diagram below:
\begin{equation*}
  \begin{tikzcd}
    &    M^\Top_i \cap M^\Bot_i \ar[dl, "\sigma_i^\Top"'] \ar[dr, "\sigma_i^\Bot"] &\\
    \dom \psi_i  \ar[rr, "\psi_i"] & & \operatorname{im} \psi_i.
  \end{tikzcd}
\end{equation*}
Commutativity follows from our expressions for $\sigma_i^\bullet$, and the fact that if $[x,s,t] \in M^\Top_i \cap M^\Bot_i$, then $t \not \in [1,2]$. While \emph{a priori} the intersection $M_i^\Top \cap M^\Bot$ induces another non-trivial cocycle element, in fact $M_i^\Top \cap M^\Bot = M_i^\Top \cap M_i^\Bot$ and $\sigma^\Bot = \sigma_i^\Bot$ on this intersection.

The holonomy pseudogroup of $(M, \mathcal{F})$ is precisely the pseudogroup generated by the elements of the Haefliger cocycle. But these are simply the $\psi_i$, and this proves that the holonomy pseudogroup of $M$ is $\Psi$.
The fiber of $\sigma^\Bot$ over $y \in N$ is $\pi(\{x\} \times \R \times (0,1))$, which is connected. Similarly, the $s_i^\bullet$ have connected fibers. By \cite[Remark 2.7 (2)]{MM03}, if $N$ is a Riemannian manifold and each $\psi_i$ is a local isometry, then $(M, \mathcal{F})$ admits a transverse metric, i.e.\ $(M, \mathcal{F})$ is a Riemannian foliation.
\end{proof}

This concludes the proof of Proposition \ref{prop:quastofol}. It is now fairly immediate that every diffeological quasifold is a leaf space of some foliation.
\begin{cor}\label{cor:quas-is-leaf-space}
  Suppose $X$ is a diffeological quasifold. Then $X$ is diffeomorphic to the leaf space of some foliation $(M, \mathcal{F})$.
\end{cor}
\begin{proof}
  Given $X$, choose a countable atlas and take $V$ and the pseudogroup $\Psi$ from Proposition \ref{prop:quastopseudo}. Then take $(M, \mathcal{F})$ from Proposition \ref{prop:quastofol}. By construction, we have
  \begin{equation*}
    X \cong V/\Psi \cong M/\mathcal{F}.
  \end{equation*}
\end{proof}

The last clause of Proposition \ref{prop:quastofol} ensures that if $V$ has a Riemannian metric and $\Psi$ acts by local isometries, then the resulting foliation is Riemannian. This lets us refine the above corollary.

\begin{cor}\label{cor:riemannian-quas-is-leaf-space}
  Suppose $X$ is a diffeological quasifold, and $\mathcal{A} = \{F_i:U_i \to V_i/\Gamma_i\}$ is a countable atlas. If each $V_i$ carries a Riemannian metric such that the $\Gamma_i$ act by isometries, and if the associated pseudogroup $\Psi$ consists entirely of local isometries, then the foliation $(M, \mathcal{F})$ from the previous corollary is a Riemannian foliation. 
\end{cor}

This suggests a possible definition of a Riemannian quasifold. We will end with the proof of Lemma \ref{lem:godemont} required for Proposition \ref{prop:quastofol}.

\begin{lem}
  \label{lem:godemont}
    The quotient space $M:= V/{\sim}$ has a unique manifold structure for which the quotient map is a smooth submersion.
  \end{lem}
  \begin{proof}
    We use Godemont's criteria to show $V/{\sim}$ is a manifold. In other words, letting $R$ denote the graph of ${\sim}$ in $V \times V$, we must show that $R$ is a closed submanifold of $V \times V$ and that the second projection $V \times V \to V$ restricts to a submersion $R \to V$.

    Let us first show that $R$ is a submanifold, and the second projection $R \to V$ is a submersion.
  \begin{itemize}
  \item At points $((x_0,s_0,t_0), (\psi_i(x_0), s_0, t_0-2)) \in R$, take the open neighbourhood
    \begin{equation*}
      (\dom \psi_i \times (i,i+1) \times (2,3)) \times (\operatorname{im} \psi_i \times (i,i+1) \times (0,1)) \subseteq V \times V.
    \end{equation*}
    The intersection of $R$ with this neighbourhood is precisely the graph of the function
    \begin{equation}\label{eq:chart1}
      (\dom \psi_i \times (i,i+1) \times (2,3)) \to V, \quad (x,s,t) \mapsto (\psi_i(x), s, t-2).
    \end{equation}
The usual parametrization of the graph of \eqref{eq:chart1} is a chart of $R$. In this chart, the second projection $R \to V$ is precisely the map \eqref{eq:chart1}, which is a submersion. A similar procedure works at points $((\psi_i(x_0), s_0, t_0-2), (x_0,s_0, t_0)) \in R$.
  \item At points $(p_0,p_0) \in R$, we consider two cases. If $t_0 \geq 1$, then necessarily $y_0 \in \dom(\psi_i)$ and $t_0 \in (i,i+1)$ for some $i$. We take an open neighbourhood
    \begin{equation*}
      (\dom \psi_i \times (i,i+1) \times J)^2,
    \end{equation*}
    (the exponent denotes the Cartesian product of the set with itself) where $J$ is a open interval in $(0,3)$ about $t_0$ chosen small enough so that it does not intersect both $(2,3)$ and $(0,1)$. The intersection of $R$ with this neighbourhood is the graph of the identity over $(\dom \psi_i \times (i,i+1) \times J)$, and the parametrization of this graph gives a submanifold chart. In these coordinates, the second projection $R \to V$ is the identity, hence a submersion.

    If $t_0 < 1$, then the intersection of $R$ and the open neighbourhood
    \begin{equation*}
      (N \times \R \times (0,1))^2
    \end{equation*}
is the graph of the identity over $N \times \R \times (0,1)$. Again, in the induced coordinates, the second projection $R\to V$ is a submersion.
  \end{itemize}
  Now we show that $R$ is closed. Suppose that $\{(p_j, p_j')\}_{j=1}^\infty$ is a sequence in $R$ with limit point $(p_0, p_0')$ in $V \times V$. Because $(p_j,p_j') \in R$, we must have $s_j' = s_j$ for all $j$, and so $s_0' = s_0$. We must show $(p_0,p_0') \in R$; we consider three cases.
  \begin{itemize}
  \item If $s_0 < 1$, then $s_j < 1$ for sufficiently large $j$, and necessarily $x_j' = x_j$ and $t_j' = t_j$. The limit is $(p_0,p_0) \in R$.
  \item If $s_0 = i$ for some $i = 1,2,\ldots $, then both $t_0$ and $t_0'$ are in $(0,1)$. Thus $t_j$ and $t_j'$ are in $(0,1)$ for large $j$, and necessarily $x_j' = x_j$ and $t_j' = t_j$. The limit point is $(p_0,p_0) \in R$.
  \item If $s_0 \in (i,i+1)$, there are three sub-cases.
    \begin{itemize}
    \item  If both $t_0 \leq 1$ and $t_0' \leq 1$, then eventually $t_j$ and $t_j'$ are strictly less than $2$, and necessarily $x_j' = x_j$ and $t_j' = t_j$. The limit is $(p_0,p_0) \in R$.
    \item If $1 < t_0 < 2$, then $1 < t_j < 2$ for large $j$, and necessarily $x_j' = x_j$ and $t_j' = t_j$. The limit is $(p_0,p_0) \in R$. The case $1 < t_0' < 2$ is similar.
    \item If $t_0 \geq 2$, then $t_j > 1$ for large $j$, so the point $(x_j', s_j', t_j')$ is either
          \begin{equation*}
      (x_j, s_j, t_j) \text{ or } (\psi_i(x_j), s_j, t_j-2).
    \end{equation*}
    The possible limit points are therefore either
    \begin{equation*}
      (p_0,p_0) \text{ or } (p_0, (\psi_i(x_0), s_0, t_0-2)).
    \end{equation*}
Whenever these are in $V\times V$, they are in $R$ too. The case $t_0' \geq 2$ is similar.
\end{itemize}
\end{itemize}
\end{proof}

  \bibliographystyle{amsalpha}
  \bibliography{riemannian-quasifolds-submission}
  
\end{document}